\documentclass[12pt]{amsart}

\usepackage{amssymb, amsfonts}
\usepackage{url}
\usepackage[all,arc]{xy}
\usepackage{amscd}
\usepackage{mathrsfs}
\usepackage{geometry}
\usepackage[colorlinks,citecolor=blue]{hyperref}
\usepackage{comment}
\usepackage{enumerate}
\usepackage{graphicx}
\usepackage{bm}
\usepackage{color}
\usepackage{anyfontsize}
\usepackage{caption}

\numberwithin{equation}{section}

\theoremstyle{plain}
\newtheorem{theorem}{Theorem}[section]
\newtheorem{corollary}[theorem]{Corollary}
\newtheorem{lemma}[theorem]{Lemma}
\newtheorem{proposition}[theorem]{Proposition}

\newtheorem{definition}[theorem]{Definition}

\theoremstyle{remark}
\newtheorem{remark}{Remark}[section]
\newtheorem{example}[remark]{Example}

\begin{document}

%\date{\today} 

\title{On electrostatic manifolds with boundary}

\author{Stanislav Demurov and Vladimir Medvedev}

\address{Faculty of Mathematics, National Research University Higher School of Economics, 6 Usacheva Street, Moscow, 119048, Russian Federation}

\email{vomedvedev@hse.ru, sademurov@edu.hse.ru}

%\thanks{}

%\subjclass{}

\begin{abstract}
Static manifolds with boundary were recently introduced in the context of the problem of prescribing both the scalar curvature of a Riemannian manifold and the mean curvature of its boundary. These manifolds are also of interest from the perspective of general relativity. In this article, we introduce electrostatic manifolds with boundary as a natural generalization of static manifolds with boundary in the presence of a nonzero electric field. We also introduce Robin electrostatic quadruples, extending the notion of Robin static triples previously introduced by the second author. We establish several uniqueness theorems for suitable regions of the Reissner--Nordstr\"om manifold. Finally, we characterize the criticality of the first eigenvalue of an associated mixed Robin--Dirichlet Schr\"odinger operator under admissible volume-preserving variations in terms of the equality of the surface gravities along the Dirichlet boundary components.
\end{abstract}

\maketitle

%%%%%%%%%%%%%%%%%%%%%%%%%%%%%%%%%%%%%%%%%%%%%%%%%%%%%%%%%%%%%%%%%%%%%%%%%
% Macros
%%%%%%%%%%%%%%%%%%%%%%%%%%%%%%%%%%%%%%%%%%%%%%%%%%%%%%%%%%%%%%%%%%%%%%%%%

\newcommand\cont{\operatorname{cont}}
\newcommand\diff{\operatorname{diff}}

\newcommand{\dvol}{\text{dA}}
\newcommand{\Ric}{\operatorname{Ric}}
\newcommand{\Hess}{\operatorname{Hess}}
\newcommand{\GL}{\operatorname{GL}}
\newcommand{\myO}{\operatorname{O}}
\newcommand{\myP}{\operatorname{P}}
\newcommand{\eye}{\operatorname{Id}}
\newcommand{\myF}{\operatorname{F}}
\newcommand{\Vol}{\operatorname{Vol}}
\newcommand{\odd}{\operatorname{odd}}
\newcommand{\even}{\operatorname{even}}
\newcommand{\ol}{\overline}
\newcommand{\mye}{\operatorname{E}}
\newcommand{\myo}{\operatorname{o}}
\newcommand{\myt}{\operatorname{t}}
\newcommand{\irr}{\operatorname{Irr}}
\newcommand{\mydiv}{\operatorname{div}}
\newcommand{\curl}{\operatorname{curl}}
\newcommand{\re}{\operatorname{Re}}
\newcommand{\im}{\operatorname{Im}}
\newcommand{\can}{\operatorname{can}}
\newcommand{\scal}{\operatorname{scal}}
\newcommand{\tr}{\operatorname{trace}}
\newcommand{\sgn}{\operatorname{sgn}}
\newcommand{\SL}{\operatorname{SL}}
\newcommand{\myspan}{\operatorname{span}}
\newcommand{\mydet}{\operatorname{det}}
\newcommand{\SO}{\operatorname{SO}}
\newcommand{\SU}{\operatorname{SU}}
\newcommand{\specl}{\operatorname{spec_{\mathcal{L}}}}
\newcommand{\fix}{\operatorname{Fix}}
\newcommand{\id}{\operatorname{id}}
\newcommand{\grad}{\operatorname{grad}}
\newcommand{\singsup}{\operatorname{singsupp}}
\newcommand{\wave}{\operatorname{wave}}
\newcommand{\ind}{\operatorname{ind}}
\newcommand{\mynull}{\operatorname{null}}
\newcommand{\mycurl}{\operatorname{curl}} % curl
\newcommand{\inj}{\operatorname{inj}}
\newcommand{\arcsinh}{\operatorname{arcsinh}}
\newcommand{\Spec}{\operatorname{Spec}}
\newcommand{\Ind}{\operatorname{Ind}}
\newcommand{\Nul}{\operatorname{Nul}}
\newcommand{\Tr}{\operatorname{Tr}}
\newcommand{\Div}{\operatorname{Div}}
\newcommand{\inrad}{\operatorname{inrad}}
\newcommand{\mult}{\operatorname{mult}}
\newcommand{\Length}{\operatorname{Length}}
\newcommand{\Area}{\operatorname{Area}}
\newcommand{\Ker}{\operatorname{Ker}}
\newcommand{\floor}[1]{\left \lfloor #1  \right \rfloor}
\newcommand{\Int}{\operatorname{Int}}

%%%%%%%%%%%%%%%%%%%%%%%%%%%%%%%%%%%%%%%%%%%%%%%%%%%%%%%%%%%%%%%%%%%%%%%%%
\section{Introduction}
An \emph{electrostatic manifold with boundary} is a quadruple $(M^n,g,V,E)$, $n\geqslant3$, where $(M^n,g)$ is a Riemannian manifold with boundary, $V$ is a nontrivial smooth function called the \emph{static potential}, and $E$ is a smooth vector field called the \emph{electric field}, satisfying, for a constant $\Lambda\in\mathbb R$,
\begin{align}
    &\Hess_gV=V\left(\Ric_g+2E^\flat\otimes E^\flat-\dfrac{2}{n-1}\bigl(|E|_g^2+\Lambda\bigr)g\right) &&\text{in }M,\tag{E1}\label{eq:electrostatic:a}\\[\medskipamount]
    &\Delta_gV=\dfrac{2}{n-1}\left((n-2)|E|_g^2-\Lambda\right)V &&\text{in }M,\tag{E2}\label{eq:electrostatic:b}\\[\medskipamount]
    &\mydiv_gE=0,\qquad d(VE^\flat)=0 &&\text{in }M,\tag{E3}\label{eq:electrostatic:c}\\[\medskipamount]
    &(\partial_\nu V)\gamma_g-VB_g=0 &&\text{on }\partial M.\tag{E4}\label{eq:electrostatic:d}
\end{align}
Here $E^\flat=g(E,\cdot)$, $\gamma_g$ is the metric induced on $\partial M$, and $B_g$ and $H_g$ are its second fundamental form and mean curvature with respect to the outward unit normal $\nu$. The constant $\Lambda$ is the \emph{cosmological constant}. The case $E=0$ recovers the static manifolds with boundary introduced in~\cite{cruz2019prescribing,almaraz2020mass,AlmarazDeLima2025}.

As in the static setting (see, for example, \cite[Proposition~2.2(d)]{cruz2023critical} with $\kappa=\tau=0$), these structures admit a variational characterization. For fixed data $(g_o,V,E)$, define the functional
\begin{equation}\label{eq:functional}
\mathcal F[g]=\int_M V\left(R_g-6|E|_g^2\right)\,dv_{g_o}
+4\int_M V|E|_g^2\,dv_g+2\int_{\partial M}VH_g\,ds_{g_o}
\end{equation}
on the space $\mathcal R(M)$ of Riemannian metrics. Theorem~\ref{thm:mathcalF} below characterizes electrostatic metrics as critical points of this functional. Their boundary and zero set have special geometric properties recorded in Lemma~\ref{lem:properties}. Unlike in the static case, the scalar curvature need not be constant; as we show below in~\eqref{eq:electrostatic 2 non-zero:a}, it is constant precisely when $|E|_g$ is constant.

Canonical examples arise from the \emph{Reissner--Nordstr\"om data}. For mass $m>0$ and charge $q\in\mathbb R$, set $\mathfrak C_n=\sqrt{2(n-2)/(n-1)}$ and define
\begin{equation}\label{eq:rn-data}
\begin{gathered}
g_{m,q}=V_{m,q}(r)^{-2}dr^2+r^2g_{\mathbb S^{n-1}},\qquad
V_{m,q}(r)=\sqrt{1-\dfrac{2m}{r^{n-2}}+\dfrac{q^2}{r^{2(n-2)}}},\\
E_{m,q}(r)=\dfrac{n-2}{\mathfrak C_n}\dfrac{q}{r^{n-1}}V_{m,q}(r)\partial_r,
\end{gathered}
\end{equation}
where $g_{\mathbb S^{n-1}}$ is the standard spherical metric. These data are defined on radial intervals where $V_{m,q}^2>0$. For $m\geqslant|q|$, the outer horizon radius is
\begin{equation}\label{eq:rn-horizon}
r_h=\left(m+\sqrt{m^2-q^2}\right)^{1/(n-2)},
\end{equation}
and the exterior is $RN^n=(r_h,\infty)\times\mathbb S^{n-1}$ with metric $g_{m,q}$. The cases $m>|q|$ and $m=|q|\neq0$ are called \emph{sub-extremal} and \emph{extremal}, respectively. In the sub-extremal case, the metric completion adds the horizon $\Sigma_{r_h}$ as a smooth boundary.

The outer photon-sphere radius is
\begin{equation}\label{eq:photon_sphere:a}
r_{\mathrm{ph}}=r_{\mathrm{ph}}^+=\left(\dfrac{nm+\sqrt{n^2m^2-4(n-1)q^2}}{2}\right)^{1/(n-2)}.
\end{equation}
Appendix~\ref{ap:rps} verifies the Robin boundary condition~\eqref{eq:electrostatic:d} on $\Sigma_{r_{\mathrm{ph}}}$ (for the general notion of a photon sphere, see~\cite{claudel2001geometry,perlick2005totally}). Thus the region
\[
RN^n_- = \bigl([r_{\mathrm{ph}},\infty)\times\mathbb S^{n-1},g_{m,q},V_{m,q},E_{m,q}\bigr)
\]
is an electrostatic manifold with boundary. In the sub-extremal case, doubling the region between $r_h$ and $r_{\mathrm{ph}}$ across the horizon gives a compact example. Doubling the entire exterior across the horizon and retaining the region bounded by the photon sphere in the reflected copy gives the noncompact manifold denoted by $RN^n_+$. The two regions are illustrated in Figure~\ref{fig:rn-regions}.

\begin{figure}[tbp]
\centering
\includegraphics[width=0.85\textwidth]{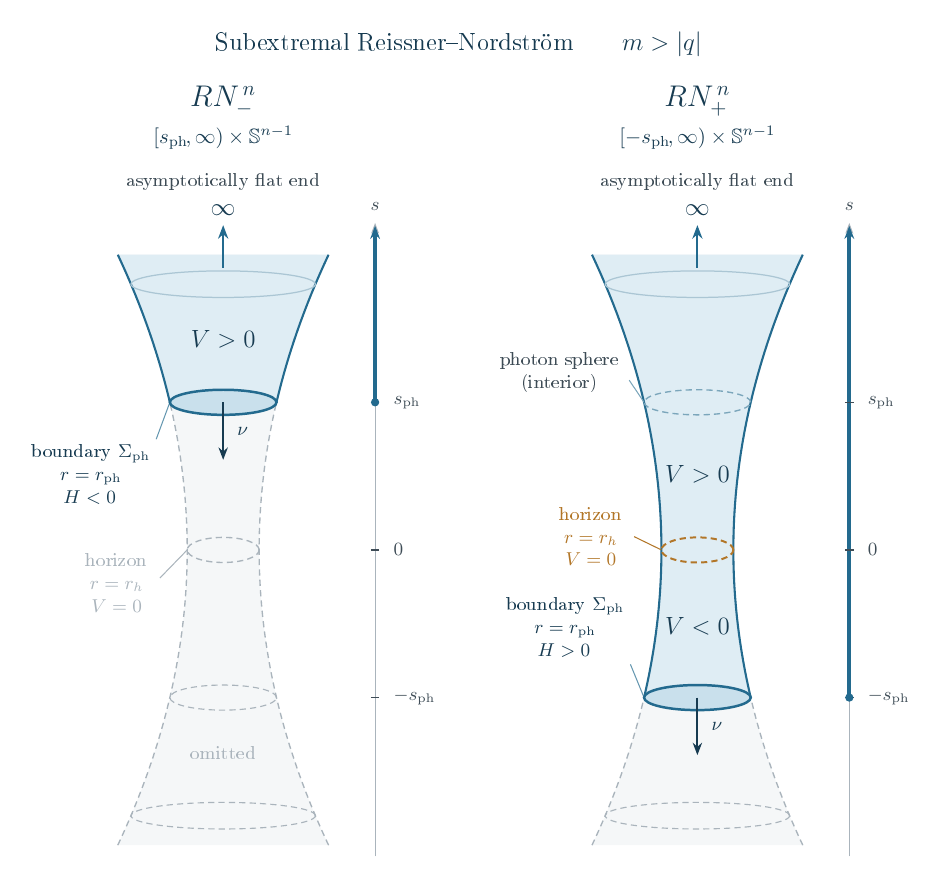}
\caption{The sub-extremal regions $RN^n_-$ and $RN^n_+$. The shaded portions are retained, and the solid cross-sections are Robin boundaries; the arrows $\nu$ indicate their outward normals. In $RN^n_+$, the compact region with $V<0$ lies between the Robin boundary and the interior horizon $V=0$. Here $s$ denotes signed arclength from the horizon, increasing toward the retained asymptotically flat end, and $s_{\mathrm{ph}}>0$ is the arclength of the outer photon sphere. This $s$ is distinct from the isotropic coordinate used in~\eqref{eq:rn-isotropic-factor}. The profiles are schematic.}
\label{fig:rn-regions}
\end{figure}

\begin{remark}
    The \emph{super-extremal} Reissner--Nordstr\"om manifold, where $m<|q|$, requires a more delicate introduction. It is discussed in greater detail in Appendix~\ref{sec:appendix}.
\end{remark}

If $(M^n,g,V,E)$ is an asymptotically flat electrostatic manifold, then it is natural to assume that there exists a smooth function $\Psi$ on $M^n$ such that $VE=-\nabla^g\Psi$ (see Section~\ref{sec:asymt_flat}). Under this assumption and by analogy with the static case (see Theorem 1.18 in~\cite{medvedev2024static}), we prove the following theorem

\begin{theorem}\label{thm:non-compact case ARNS}
Let $(M^n,g,V, E)$ be a complete (up to the boundary) one-ended asymptotically Reissner--Nordstr\"om system of mass $m$ and charge $q$ (see Section~\ref{sec:asymt_flat}), which is an electrostatic manifold with compact boundary. Then
\begin{enumerate}
    \item[I.]
If all the connected components of $\partial M$ are sub-extremal (i.e. satisfying \eqref{eq:quasiloc. subextremal}), $V\neq 0$ on $M\cup \partial M$, $H_g<0$, $V=const.$ on $\partial M$, $E\in\Gamma(N\partial M)$, and $|E| = const.$ on $\partial M$, then $(M,g)$ is isometric to $RN^n_{-}$ with $m>|q|$, $V=V_{m,q}$, $E = E_{m,q}$. 
    \item[II.]
If $\partial M$ is connected and sub-extremal, $n=3$, $\Sigma=V^{-1}(0)\subset Int\, M$ is compact and separates the boundary and the end, $\nu(V)=const.\neq 0$ on $\Sigma$, $V=const.\neq 0$ on $\partial M$, $E\in\Gamma(N\partial M)$, $|E| = const.$ on $\partial M$, and $H_g>0$, then $(M,g)$ is isometric to $RN^3_{+}$ with $m>|q|$, $V=V_{m,q}$ and $E=E_{m,q}$. 
\end{enumerate}

Here the model fields are written in isotropic coordinates $s$, related to the radial coordinate by $r=s\varphi_{m,q}(s)$, where
\begin{equation}\label{eq:rn-isotropic-factor}
\varphi_{m,q}(s)=\left(1+\dfrac{m+q}{2s^{n-2}}\right)^{1/(n-2)}
\left(1+\dfrac{m-q}{2s^{n-2}}\right)^{1/(n-2)}.
\end{equation}
In these coordinates,
\begin{equation}\label{eq:rn-isotropic-fields}
\begin{aligned}
V_{m,q}(s)&=\dfrac{1-\dfrac{m^2-q^2}{4s^{2(n-2)}}}
{\left(1+\dfrac{m+q}{2s^{n-2}}\right)\left(1+\dfrac{m-q}{2s^{n-2}}\right)},\\[\medskipamount]
E_{m,q}(s)&=\dfrac{n-2}{\mathfrak C_n}
\dfrac{q\left(1-\dfrac{m^2-q^2}{4s^{2(n-2)}}\right)}
{s^{n-1}\varphi_{m,q}^{2n-3}(s)\left(\varphi_{m,q}(s)+s\varphi_{m,q}'(s)\right)}\partial_s.
\end{aligned}
\end{equation}
\end{theorem}

Also, using the results from \cite{borghini2024asymptflatelctro}, we obtain the following result for the 3-dimensional case with weaker assumptions on asymptotic decay

\begin{theorem}\label{thm:non-compact_2}
Let $(M^3,g,V,E)$ be a complete (up to the boundary) one-ended asymptotically flat electrostatic system (see Section~\ref{sec:asymt_flat}), which is an electrostatic manifold with compact connected boundary.  Then, assume that on $\partial M$ the inequality $V^2\geqslant|1-\Psi^2|$ holds if $V^2>(1-|\Psi|)^2$ and that $V=1$ and $\Psi=0$ do not both hold. If $V\neq 0$ on $M\cup \partial M$, $H_g<0$, $V=const.$ on $\partial M$, $E\in\Gamma(N\partial M)$, and $|E| = const.$ on $\partial M$, then $(M,g)$ is isometric to $RN^3_{-}$, $V=V_{m,q}$, $E = E_{m,q}$ with $m>|q|$ or $m=|q|\neq 0$ or $m<|q|$. 
\end{theorem}

\begin{remark}
It is conceivable that Theorem~\ref{thm:non-compact case ARNS} and Theorem~\ref{thm:non-compact_2} can be proved with fewer assumptions in the class of asymptotically flat spin static manifolds, as achieved in \cite[Theorem 6]{raulot2025admcapacitors}.
\end{remark}

Motivated by the Robin static triples introduced by the second author in~\cite{medvedev2025some}, which encompass both static triples and static manifolds with boundary, we introduce the following electrostatic analogue.

\begin{definition}\label{def:robin-electrostatic}
Let $M^n$, $n\geqslant3$, be a compact connected manifold with boundary, decomposed as $\partial M=D\cup S$. The faces $D$ and $S$ are allowed to be empty individually and have disjoint relative interiors. They may meet along a common smooth edge; apart from this edge, the boundary is smooth. All the data below are assumed smooth up to the faces and their intersection.

A quadruple $(M,g,V,E)$ is called a \emph{Robin electrostatic quadruple with cosmological constant $\Lambda\in\mathbb R$} if $E$ is a smooth vector field, $V=0$ on $D$, $V>0$ on $M\setminus D$, and
\begin{align}
\Hess_gV&=V\left(\Ric_g+2E^\flat\otimes E^\flat-\dfrac{2}{n-1}\bigl(|E|_g^2+\Lambda\bigr)g\right)&&\text{in }M,\label{eq:electrostatic-hessian}\\
\Delta_gV&=\dfrac{2}{n-1}\left((n-2)|E|_g^2-\Lambda\right)V&&\text{in }M,\label{eq:electrostatic-trace}\\
\mydiv_gE&=0,\qquad d(VE^\flat)=0&&\text{in }M,\label{eq:maxwell}\\
(\partial_\nu V)\gamma_g-VB_g&=0&&\text{on }S.\label{eq:electrostatic-robin}
\end{align}
The set $D$ is the Dirichlet face and $S$ is the Robin face. The case $S=\varnothing$ is the purely Dirichlet case.
\end{definition}

In~\cite{medvedev2026spectral}, the second author characterized compact Robin static triples whose metrics are critical for the first mixed Robin--Dirichlet eigenvalue under admissible volume-preserving deformations. We extend this characterization to the newly introduced class of Robin electrostatic quadruples.

\begin{theorem}\label{prop:electrostatic-criticality}
Let $(M^n,g,V,E)$ be a Robin electrostatic quadruple with $D\neq\varnothing$. Write
\[
D=\Sigma_1\sqcup\cdots\sqcup\Sigma_\ell,\qquad H_D^1(M,dv_g)=\{w\in H^1(M,dv_g):\Tr_Dw=0\}.
\]
For a smooth metric $\widetilde g$ and a smooth vector field $\widetilde E$ on the fixed manifold $M$, define
\begin{equation}\label{eq:spectral-functional}
\mathcal F(\widetilde g,\widetilde E)=\inf_{w\in H_D^1(M,dv_{\tilde g})\setminus\{0\}}
\dfrac{\displaystyle\int_M\left(|\nabla^{\widetilde g}w|_{\widetilde g}^2+\dfrac{2(n-2)}{n-1}|\widetilde E|_{\widetilde g}^2w^2\right)dv_{\widetilde g}-\dfrac1{n-1}\int_S H_{\widetilde g}w^2\,ds_{\widetilde g}}{\displaystyle\int_Mw^2\,dv_{\widetilde g}}.
\end{equation}
If $S=\varnothing$, the boundary integral is omitted. This is the first eigenvalue of 
$$
-\Delta_{\widetilde g}+\dfrac{2(n-2)}{n-1}|\widetilde E|_{\widetilde g}^2
$$ 
with Dirichlet conditions on $D$ and Robin conditions $\partial_{\nu_{\widetilde g}}w=H_{\widetilde g}w/(n-1)$ on $S$.

Call a smooth variation $(g_t,E_t)$ \emph{admissible} if
\begin{gather}
(g_0,E_0)=(g,E),\qquad\Vol_{g_t}(M)=\Vol_g(M),\label{eq:admissible-volume}\\
\dot g_0=\mathcal L_Xg,\qquad\dot E_0=\mathcal L_XE=[X,E],\label{eq:admissible-derivatives}
\end{gather}
for some smooth vector field $X$ tangent to $S$. There is no tangency requirement along $D$. When $S=\varnothing$, there is no boundary tangency requirement on $X$.

Then
\begin{equation}\label{eq:principal-electrostatic-eigenvalue}
\mathcal F(g,E)=\dfrac{2\Lambda}{n-1},
\end{equation}
and each $\kappa_j=|\nabla^gV|_g\big|_{\Sigma_j}$ is a positive constant. Moreover, the following are equivalent:
\begin{enumerate}
\item $\left.\dfrac{d}{dt}\right|_{t=0}\mathcal F(g_t,E_t)=0$ for every admissible variation;
\item $\kappa_1=\cdots=\kappa_\ell$.
\end{enumerate}
In particular, connectedness of $D$ implies criticality. The varied pairs $(g_t,E_t)$ are not required to satisfy the electrostatic equations.
\end{theorem}

All manifolds in this article are supposed to be orientable.

\subsection{Paper organization} The paper is organized as follows. In Section~\ref{sec:prelim}, we derive the defining equations of electrostatic manifolds from the source-free Einstein--Maxwell equations for static spacetimes and show that metrics on electrostatic manifolds with boundary arise as critical points of an appropriate variational functional. In Section~\ref{sec:prop}, we establish a key lemma concerning general properties of electrostatic manifolds with boundary. Section~\ref{sec:non-compact} provides the necessary background for Theorems~\ref{thm:non-compact case ARNS} and~\ref{thm:non-compact_2}, and contains their proofs. Section~\ref{sec:robin-electrostatic-quadruples} is devoted to Robin electrostatic quadruples. There, we prove Theorem~\ref{prop:electrostatic-criticality}, derive a corollary, and present examples illustrating the preceding results. Finally, in Appendix~\ref{sec:appendix}, we collect auxiliary computations used throughout the paper.

\subsection*{Acknowledgments} This article is an output of a research project (HSE-BR-2025-059) implemented as part of the Basic Research Program at the National Research University Higher School of Economics (HSE University).

\subsection*{AI disclosure} The work presented in this article was initiated in 2025, prior to the recent surge in the use of AI tools in mathematical research. In particular, Sections~\ref{sec:prelim}--\ref{sec:non-compact} and~\ref{sec:appendix} were written in 2025. Although Section~\ref{sec:robin-electrostatic-quadruples} was written more recently, all mathematical ideas, proofs, and computations throughout the article were developed and carried out independently by the authors. OpenAI's ChatGPT was used to assist with organizing the text, improving the exposition, language editing, elaborating on the authors' original computations, identifying relevant references, and producing illustrations. The authors take full responsibility for the mathematical content and accuracy of the article.

\section{Preliminaries}\label{sec:prelim}

For completeness, we recall the standard reduction of the source-free Einstein--Maxwell equations to the electrostatic system; see Appendix A of~\cite{chrusciel2017non} and Section 2.1 of~\cite{cruz2024minmax}. We then adapt the variational framework for static manifolds with boundary to incorporate the electric field and the Robin boundary condition.

\begin{definition}
The \emph{source-free Einstein--Maxwell equations}  for $(n+1)$-dimensional spacetime $\displaystyle \left(\mathbb{R} \times M^n, \mathfrak{g}\right)$ are
\begin{align}%\label{eq:einst--maxwell}
    & \Ric_{\mathfrak g} - \dfrac{1}{2}R_{\mathfrak g}\mathfrak g +\Lambda \mathfrak g = 2\left(F\circ F - \dfrac{1}{4}|F|_{\mathfrak g}^2 \, \mathfrak g\right), \tag{EM1}\label{eq:einst--maxwell:a} \\[\medskipamount]
    & dF=0, \quad d*_{\mathfrak g} F = 0. \tag{EM2}\label{eq:einst--maxwell:b}
\end{align}
\noindent where $F$ is the \emph{electromagnetic tensor} and $\displaystyle \left(F \circ F\right)_{\alpha\beta}=\mathfrak g^{\mu\nu}F_{\alpha\mu}F_{\beta\nu}$. The solutions of these equations are called \emph{electro-vacuum spacetimes}.
\end{definition}

Taking the metric trace, we get 
\begin{equation}\label{eq:contracted einst.}
        R_{\mathfrak g} = \dfrac{n-3}{n-1}|F|_{\mathfrak g}^2 + \dfrac{2(n+1)}{n-1}\Lambda.
    \end{equation}
    
    Inserting \eqref{eq:contracted einst.} into \eqref{eq:einst--maxwell:a}, we get an alternative version of the Einstein--Maxwell equations for $(n+1)$-dimensional spacetime
\begin{equation}\tag{AEM1}\label{eq:einst--maxwell 2:a}
\Ric_{\mathfrak g}=2\left(F\circ F-\dfrac{1}{2(n-1)}|F|_{\mathfrak g}^2\,\mathfrak g\right)+\dfrac{2}{n-1}\Lambda\mathfrak g,
\end{equation}
together with the unchanged Maxwell equations~\eqref{eq:einst--maxwell:b}.

\begin{definition}\label{def:static spacetime}
A \emph{static spacetime} is a Lorentzian manifold which can be locally represented as the Lorentzian warped product of a complete Riemannian manifold $\left(M, g\right)$ with the real line, i.e., there exists a warping function $V \in C^\infty(M)$ such that the Lorentzian metric takes the form $\mathfrak g=-V^2dt^2 + g$ on $\mathbb{R} \times M^n$.
\end{definition}

Measured by the static observer $\displaystyle \eta=V^{-1}\partial_t$ the electromagnetic tensor $F$ can be uniquely decomposed in terms of the \emph{electric field} $E$ and the \emph{magnetic field $B$} (see \cite[Section 6.4.1]{gourgoulhon_2012_3+1} for $(3+1)$-spacetime). We will make the typical assumption of the vanishing of the magnetic field. Then, the decomposition takes the following form:
\begin{align*}%\label{eq:electromagnetic tensor}
    &E^\flat = -\iota_\eta F, \quad F = E^\flat\wedge \eta^\flat = VE^\flat \wedge dt,
\end{align*}
\noindent where $\iota$ denotes the interior multiplication on tensors, and the equations imply that $F_{ti} = - F_{it}= VE_i^\flat$ and $F_{ij}=0$.

Let us calculate how equations \eqref{eq:einst--maxwell:b} simplify when working in a static spacetime with electromagnetic tensor $F$.

Firstly, we have
\begin{align*}
    0 = dF &= d\left(VE^\flat \wedge dt\right) = d\left(VE^\flat\right)\wedge dt - VE^\flat\wedge ddt = d\left(VE^\flat\right)\wedge dt \quad\Longrightarrow\\[\medskipamount]
    &\Longrightarrow\quad d\left(VE^\flat\right) = 0.
\end{align*}

Secondly, choose an orthonormal coframe $\{\varepsilon^i\}, i = \overline{1,n}\,$ for $g$ and define $\varepsilon^t = V dt$. Then, for $i,j \ne t$
\begin{align*}
     g = \delta_{ij} \, \varepsilon^i \otimes \varepsilon^j , \quad \mathfrak g = -\varepsilon^t \otimes \varepsilon^t + \delta_{ij} \, \varepsilon^i \otimes \varepsilon^j.
\end{align*}

Recall that $\displaystyle *_g(\varepsilon^{i_1} \wedge ...\wedge \varepsilon^{i_k}) = \pm\varepsilon^{j_k} \wedge ... \wedge \varepsilon^{j_{n-k}}$, where $\displaystyle (i_1, ..., i_k, j_1, ..., j_{n-k})$ is some permutation of $(1, 2,  ..., n)$. Then
\begin{equation*}
    *_\mathfrak{g}(\varepsilon^{i_1}\wedge \varepsilon^t) = \pm \widehat{\varepsilon^{t}} \wedge \varepsilon^{j_1}\wedge...\wedge\varepsilon^{j_{n-1}} = \pm*_g(\varepsilon^{i_1}), 
\end{equation*}
\noindent where $\widehat{...}$ denotes the missing term.

Thus, since $\displaystyle \mydiv E = *d*E^\flat$, we have
\begin{align*}
    0 &= d*_{\mathfrak g} F = d*_{\mathfrak g} \left(VE^\flat \wedge dt\right) = d*_{\mathfrak g} \left(VE^\flat \wedge \dfrac{\varepsilon^t}{V}\right) = d*_{\mathfrak g} \left(E^\flat \wedge \varepsilon^t\right) =\\
    &= \pm d*_g E^\flat \quad\Longrightarrow\quad \mydiv_g E = 0.
\end{align*}

Therefore, equations \eqref{eq:einst--maxwell:b} are reduced to 
\begin{align}\label{eq:maxwell reduciton}
    \mydiv_g E = 0, \quad d\left(VE^\flat\right) = 0,
\end{align}
\noindent where for $n=3$ case we have $d(VE^\flat) = \mycurl(VE) = 0$.

Now, using $t$ to denote temporal components, $i,j$ to denote spatial components, and Greek letters when the nature of the component is not specified, a straightforward computation shows that 
\begin{multline}\label{eq:einst--maxwell RHS}
    \left(F \circ F\right)_{\alpha\beta} - \dfrac{1}{2(n-1)}F_{\mu\nu}F^{\mu\nu} \, \mathfrak{g}_{\alpha\beta} = \\[\medskipamount]
     = -\dfrac{1}{V^2}F_{t\alpha}F_{t\beta} + g^{ij}F_{i\alpha}F_{j\beta} + \dfrac{1}{n-1}\dfrac{1}{V^2}g^{ij}E^\flat_iE^\flat_j \, \mathfrak{g}_{\alpha\beta}, 
\end{multline}

Also, it is well known that for static metric $\Ric_\mathfrak{g}$ can be decomposed into
\begin{align}\label{eq:Ric decomposition}
    & (\Ric_\mathfrak{g})_{tt} = V\Delta_gV,\quad (\Ric_\mathfrak{g})_{tj} = 0, \quad (\Ric_\mathfrak{g})_{ij} = (\Ric_g)_{ij} - \dfrac{1}{V} \left(\Hess_g V\right)_{ij}.
\end{align}

Substituting~\eqref{eq:einst--maxwell RHS} and~\eqref{eq:Ric decomposition} into~\eqref{eq:einst--maxwell 2:a} gives the interior equations~\eqref{eq:electrostatic:a} and~\eqref{eq:electrostatic:b} stated in the introduction, while~\eqref{eq:maxwell reduciton} supplies~\eqref{eq:electrostatic:c}. This derives the interior part of the defining system from the source-free Einstein--Maxwell equations.

\begin{definition}\label{def:electrostatic}
We use the term \emph{electrostatic manifold} for a quadruple $(M^n,g,V,E)$ satisfying the interior equations~\eqref{eq:electrostatic:a}--\eqref{eq:electrostatic:c}, with the data and conventions fixed in the introduction. An \emph{electrostatic manifold with boundary} additionally satisfies~\eqref{eq:electrostatic:d}.
\end{definition}

By taking the metric trace of \eqref{eq:electrostatic:a} and \eqref{eq:electrostatic:d}, we additionally get 
\begin{align}%\label{eq:electrostatic traces}
    &\Delta_g V = V \left(R_g - \dfrac{2n}{n-1} \Lambda - \dfrac{2}{n-1}|E|_g^2\right)  &&\text{in } M, \tag{TE1}\label{eq:electrostatic traces:a}\\[\medskipamount]
    &\dfrac{\partial V}{\partial \nu} = \dfrac{H_g}{n-1}V &&\text{on } \partial M.\tag{TE2}\label{eq:electrostatic traces:b}
\end{align}

Thus, \eqref{eq:electrostatic traces:a},\eqref{eq:electrostatic traces:b}, \eqref{eq:electrostatic:b} and \eqref{eq:electrostatic:d} imply that if $V \not\equiv 0$ on an electrostatic manifold (with boundary) $(M,g,V,E)$, then the following equations hold:
\begin{align}%\label{eq:electrostatic 2 non-zero}
    & R_g = 2|E|_g^2 + 2\Lambda &&\text{in } M,\tag{NE1}\label{eq:electrostatic 2 non-zero:a}\\[\medskipamount]
    & B_g = \dfrac{H_g}{n-1}g &&\text{on } \partial M. \tag{NE2}\label{eq:electrostatic 2 non-zero:b}
\end{align}

Now, for the sake of completeness, we rewrite \cite[Proposition 5]{cruz2024minmax} and show that it is applicable to $n$-dimensional electrostatic manifolds (with boundary). 

\begin{lemma}\label{lem:electrostatic PDE}
Let $(M, g, V, E)$ be an electrostatic manifold (with boundary), then the equations \eqref{eq:electrostatic:a}, \eqref{eq:electrostatic:b} and \eqref{eq:electrostatic:c} are equivalent to the following second order overdetermined elliptic equation:
\begin{equation}\label{eq:electrostatic PDE}
    \Hess_gV - \left(\Delta_gV\right)g - V\Ric_g = 2V\left(E^\flat \otimes E^\flat - |E|_g^2 \, g\right)
\end{equation}
\noindent In particular, for the formal $L^2$-adjoint operator to the linearized scalar curvature operator $\displaystyle DR^*|_g(V) = \Hess_gV - \left(\Delta_gV\right)g - V\Ric_g$, we see that $\displaystyle V \in \Ker\left(DR^*|_g\right)$ if and only if $E \equiv 0$.
\end{lemma}

\begin{proof} The equations \eqref{eq:electrostatic:a} and \eqref{eq:electrostatic:b} imply \eqref{eq:electrostatic PDE}:
\begin{align*}
    \Hess_gV &= V\Ric_g +2V E^\flat \otimes E^\flat - V\dfrac{2}{n-1}|E|_g^2g + \left(\Delta_gV - V\dfrac{2(n-2)}{n-1} |E|_g^2\right)g = \\[\medskipamount]
    & = V\Ric_g + 2V E^\flat \otimes E^\flat + \left(\Delta_gV\right)g - 2V|E|_g^2\left(\dfrac{n-2}{n-1}+\dfrac{1}{n-1}\right)g = \\[\medskipamount]
    & = V\Ric_g + \left(\Delta_gV\right)g +2V\left(E^\flat \otimes E^\flat - |E|_g^2 \, g\right).
\end{align*}

For the proof of the converse implication, see the proof of \cite[Proposition 5]{cruz2024minmax} where the condition $\displaystyle d\left(VE^\flat\right)=0$ from \eqref{eq:electrostatic:c} is used from the beginning.
  
\end{proof}

\begin{corollary}\label{cor:electrostatic PDE}
Combining Lemma~\ref{lem:electrostatic PDE} with the boundary equation~\eqref{eq:electrostatic:d} gives the formulation of the electrostatic boundary problem by~\eqref{eq:electrostatic PDE} in $M$ and~\eqref{eq:electrostatic:d} on $\partial M$.
\end{corollary}

By taking the metric trace of the equation~\eqref{eq:electrostatic PDE} we additionally get
\begin{equation}\label{eq:electrostatic 2 contracted}
    \Delta_gV =  \left(- \dfrac{R_g}{n-1}+2|E|_g^2\right)V.
\end{equation}

\begin{theorem}\label{thm:mathcalF}
    The quadruple $\left(M^n,g_o,V,E\right)$ is an electrostatic manifold with boundary if and only if $g_o\in\mathcal R(M)$ is critical for the functional~\eqref{eq:functional} defined in the introduction, with $V$ and $E$ held fixed.
\end{theorem}

\begin{proof}
Let $\delta_g$ denote the variation with respect to the metric $g$ and $\displaystyle h = \dot{g}(t)|_{t=0}, \, g(0) = g$. Then
\begin{align}\label{eq:var1}
    \delta_g\left(|E|_g^2\right)(h) &= \delta_g\left(g_{ij} E^iE^j\right) = \dfrac{d}{dt}|_{t=0}\left(g(t)_{ij}E^iE^j\right) = \notag\\[\medskipamount]
    & = h_{ij}E^iE^j = g^{ik}g^{jl}E^\flat_kE^\flat_l h_{ij} = \notag\\[\medskipamount]
    & = \langle E^\flat \otimes E^\flat , h \rangle_g.
\end{align}

 Using the Leibniz rule, we obtain
\begin{align}\label{eq:var2}
    \delta_g\left(|E|_g^2 \,dv_g\right)(h) &= \delta_g\left(|E|_g^2\right)(h)\,dv_g + |E|_g^2 \, \delta_g\left(\,dv_g\right)(h) = \notag\\[\medskipamount]
    & = \langle E^\flat \otimes E^\flat + \dfrac{1}{2} |E|_g^2 \, g, h \rangle_g \,dv_g,
\end{align}
\noindent where we used $\displaystyle \delta_g(\,dv_g)(h) = \dfrac{1}{2} \langle g, h \rangle_g \,dv_g$.

Now denote $\gamma := g_{\partial M}, \gamma_o := (g_o)_{\partial M}$ and recall that $\displaystyle g_o\in\mathcal{R}(M)$ is critical for the functional $\mathcal{F}[g]$ $\iff$ $\delta_g\mathcal{F}(g_o)(h) = 0$. Let us calculate $\delta_g\mathcal{F}(g_o)(h)$.

 Variation of the first term in \eqref{eq:functional}:
\begin{align*}
    \delta_g  \left(\int_M V\left(R_g - 6|E|_g^2\right)\,dv_{g_o}\right)(h) &= \int_M V\delta_g (R_g \, dv_{g_o})(h) - 6\int_M V\delta_g (|E|_g^2 \,dv_{g_o})(h) = \\[\medskipamount]
    & = \int_M V \left(\delta_g R_g h\right) \,dv_{g_o}  - 6\int_M V \langle E^\flat \otimes E^\flat, h \rangle_g \,dv_{g_o} .
\end{align*}

 Variation of the second term in \eqref{eq:functional}:
\begin{align*}
    \delta_g\left(4\int_M V|E|_g^2\,dv_g\right) (h) = 4\int_M V \langle E^\flat \otimes E^\flat + \dfrac{1}{2} |E|_g^2 \, g, h \rangle_g \,dv_g .
\end{align*}

Variation of the third term in \eqref{eq:functional}:
\begin{align*}
    \delta_g\left(2\int_{\partial M} V H_\gamma \,ds_{\gamma_o}\right) (h) &= 2\int_{\partial M} V\left(\delta_\gamma H_\gamma h\right)ds_{\gamma_o}.
\end{align*}

 Summing these terms, using the calculations after Lemma 2.1 in \cite{cruz2019prescribing} and equations \eqref{eq:var1}, \eqref{eq:var2}, we obtain 
\begin{align*}
    &\delta_g\mathcal{F}(g)(h) = \int_M V \left(\delta_g R_g h \right) \,dv_{g_o}  - 6\int_M V \langle E^\flat \otimes E^\flat, h \rangle_g \,dv_{g_o}  +\\[\medskipamount]
    & \phantom{=} + 4\int_M V \langle E^\flat \otimes E^\flat + \dfrac{1}{2} |E|_g^2 \, g, h \rangle_g \,dv_g  + 2\int_{\partial M} V\left(\delta_g H_g h \right)ds_{\gamma_o} = \\[\medskipamount]
    & = \int_M V\delta_g R_g h \, dv_{g_o} + 2\int_{\partial M}V\delta_\gamma H_\gamma h \, ds_{\gamma_o} - \int_M  \langle 2VE^\flat \otimes E^\flat + 2V |E|_g^2 \, g, h \rangle_g \,dv_{g_o} = \\[\medskipamount]
    & = \int_M \left(-\Delta_g V g (\tr_gh) + \left\langle \Hess_g V, h \right\rangle_g - V\left\langle h, \Ric_g \right\rangle \right) \,dv_{g_o} +  \\[\medskipamount]
    & \phantom{=}  +\int_{\partial M} \left(\tr_\gamma h \dfrac{\partial V}{\partial\nu} - V\left\langle B_\gamma, h\right\rangle_\gamma \right) \,ds_{\gamma_o} - \int_M \langle 2VE^\flat \otimes E^\flat + 2V |E|_g^2 \, g, h \rangle_g \,dv_{g_o} = \\[\medskipamount]
    & = \int_M \Big\langle(-\Delta_g V) g + \Hess_g V - V\Ric_g - 2VE^\flat \otimes E^\flat + 2V|E|_g^2 \, g, h \Big\rangle_g \,dv_{g_o} + \\[\medskipamount]
    & \phantom{=} + \int_{\partial M} \left\langle\dfrac{\partial V}{\partial\nu} g - V B_\gamma, h\right\rangle_\gamma \,ds_{\gamma_o},
\end{align*}
where we used that $\tr_gh = \langle g, h \rangle_g$ and without loss of generality we added the assumption of working with $g_o$ for the volume form in the last integral in the third line.

This first variation vanishes for every $h$ precisely when~\eqref{eq:electrostatic PDE} and~\eqref{eq:electrostatic:d} hold. The conclusion follows from Corollary~\ref{cor:electrostatic PDE}.
\end{proof}

\begin{proposition}\label{prop:adjoint}
    A nontrivial $V \in C^\infty(M)$ is an \emph{electrostatic potential} if and only if it satisfies
    \begin{align}\label{eq:func0}
        \left(D\left(R_g - 2|E|_g^2, 2H_g\right)\right)^*(V) = \left(-2V|E|_g^2 \, g, 0\right).
    \end{align}
    The adjoint of the linearized curvature map is taken with respect to the $L^2$ pairing and evaluated at $V$. The vacuum characterization in Lemma~\ref{lem:electrostatic PDE} also holds for the boundary adjoint $\left(D(R_g,2H_g)\right)^*$.
\end{proposition}

\begin{proof}
Use the adjoint computation from Theorem~\ref{thm:mathcalF} and the variation identity~\eqref{eq:var1}. Subtracting $2|E|_g^2$ from the scalar curvature contributes $-2VE^\flat\otimes E^\flat$ to the interior component and leaves the boundary component unchanged. Equations~\eqref{eq:electrostatic PDE} and~\eqref{eq:electrostatic:d} then give~\eqref{eq:func0}.
\end{proof}

\section{General properties of electrostatic manifolds with boundary}\label{sec:prop}

We collect the standard properties of the zero set of an electrostatic potential and their extensions to the Robin boundary setting, following~\cite{cruz2024minmax} and~\cite{cruz2023static}. We also derive the boundary identity relating the tangential gradient of the mean curvature to the electric field. These preliminary results will be used in the subsequent rigidity arguments.

\begin{lemma}\label{lem:properties} 
Let $(M^n,g, V, E)$ be an electrostatic manifold with boundary and $\Sigma=V^{-1}(0)$ nonempty. Then:
\begin{itemize}
    \item[(a)] $\Sigma$ is an embedded totally geodesic hypersurface in $M$. More precisely:
    \begin{itemize}
        \item[(a.1)] If $\Sigma\cap \partial M = \varnothing$, then $\Sigma$ is a totally geodesic hypersurface contained in the interior of M, $\Int(M)$;
        \item[(a.2)] If $\Sigma\cap \partial M\neq \varnothing $, then each connected component $\Sigma_0$ of $\Sigma$ such that $\Sigma_0\cap \partial M\neq \varnothing$ is a free boundary totally geodesic hypersurface of $M$. In particular, $\partial \Sigma_0$ is a totally geodesic hypersurface of $\partial M$.
    \end{itemize}
    
    \item[(b)] The surface gravity $\kappa:=|\nabla^g V|_g$ is a positive constant
    on each connected component of $\Sigma;$
    
    \item[(c)] The electric field $E$ and $\nabla^g V$ are linearly dependent along $\Sigma$;
    
    \item[(d)] The scalar curvature is given by~\eqref{eq:electrostatic 2 non-zero:a};

    \item[(e)] On $\partial M$, we have 
    \begin{align}\label{eq:dH_on_boundary}
        dH_g=2E^\flat(\nu) E^\flat,
    \end{align}
    \noindent where $\nu$ is the unit normal vector field to $\partial M$. In particular, for the case where $E \in \Gamma(T\partial M)$ or $E \in \Gamma(N\partial M)$ (from which the latter case, for example, naturally occurs for spherically symmetric spaces) we have
    \begin{align*}
        dH_g=0 \text{ and } \Ric_g (\nu, X)=0 \quad \forall X\in \Gamma(T\partial M);
    \end{align*}

    \item[(f)] The boundary is totally umbilical by~\eqref{eq:electrostatic 2 non-zero:b}. Under either of the hypotheses on $E$ in (e), if a connected component $\Sigma_0$ of $\Sigma$ meets $\partial M$, the mean curvature of $\partial\Sigma_0$ in $\Sigma_0$ is locally constant.
    
\end{itemize}
\end{lemma}

\begin{proof}
 Items (a), (b), and (c) follow from \cite[Proposition 1]{cruz2023static}, \cite[Lemma 4]{cruz2024minmax} and \cite{aronszajn1957continuation}.

Item (d) recalls~\eqref{eq:electrostatic 2 non-zero:a}; item (f) follows from~\eqref{eq:electrostatic 2 non-zero:b}, (a.2), and (e).

To prove (e), let $p\in\partial M$, $X\in T_p\partial M$, and let $\nu$ be the outward unit normal. Then, recalling that $B_g(X,Y) = g(\nabla_X\nu, Y) \quad \forall X, Y \in T_p \partial M$, at $p \in \partial M$ we have
    \begin{align}\label{eq:Hess_X_nu_boundary}
        \Hess_g V(X, \nu) & = \nabla_X\nabla_\nu V - \nabla_{\nabla_X \nu} V = X (\nu(V)) - (\nabla_X \nu) (V) = \dfrac{X(H_g)}{n-1} V,
    \end{align}
    where we used \eqref{eq:electrostatic traces:b} and \eqref{eq:electrostatic 2 non-zero:b} for the penultimate equality, and the Leibniz rule for the ultimate equality.

    Recall the contracted Codazzi equation:
    \begin{align}\label{eq:Codazzi contracted}
        \mydiv_{\partial M} B_g - dH_g = \Ric_g(\nu, \cdot) \quad\text{on } \partial M.
    \end{align}

    Since $\partial M$ is umbilical and, in particular, $B_g=\dfrac{H_g}{n-1}g$, we have
    \begin{align}\label{eq:div_B_X}
        &(\mydiv_{\partial M} B_g) (X) =  \dfrac{X(H_g)}{n-1} \quad \forall X\in\Gamma(T\partial M).
    \end{align}

    Then, equations \eqref{eq:Codazzi contracted}, \eqref{eq:Hess_X_nu_boundary}, and \eqref{eq:div_B_X} imply 
    \begin{align}\label{eq:Ric_nu_X_boundary}
        \Ric_g(\nu, X) & = (\mydiv_{\partial M} B_g)(X) - dH_g (X) = -\dfrac{n-2}{n-1} X(H_g) \quad \forall X\in\Gamma(T\partial M).
    \end{align}
    Finally, substituting \eqref{eq:Hess_X_nu_boundary} and \eqref{eq:Ric_nu_X_boundary} into \eqref{eq:electrostatic:a} we obtain 
    \begin{align}
        X(H_g) = 2E^\flat(\nu)E^\flat(X),
    \end{align}
     which is exactly equation \eqref{eq:dH_on_boundary}.
\end{proof}

\section{Uniqueness theorems for non-compact electrostatic manifolds with boundary}\label{sec:non-compact}
In this section, we provide definitions and obtain results that are instrumental in both the statement and the proofs of Theorem~\ref{thm:non-compact case ARNS} and Theorem~\ref{thm:non-compact_2}, and then conclude by proving these theorems.

\subsection{Quasi-local photon spheres}

We begin with the definition of quasi-local photon spheres as in \cite{jahns2019photon} and \cite{cederbaum2023equi} (only without restriction of the global existence of the smooth function $\Psi$ satisfying $VE=-\nabla^g\Psi$), and then in Lemma~\ref{lem:quasiloc} we prove a key result relating the boundary of the electrostatic manifold with boundary to photon spheres, which will be instrumental in the proofs of the uniqueness theorems.
\begin{definition}\label{def:quasiloc.}
Let $(M^n,g,V,E)$ be an electrostatic manifold (with boundary). Then a connected component $S$ of a hypersurface in $M$ is called a \emph{quasi-local photon sphere} if it is totally umbilical, $H_S = const. < 0$, $V = const.$ on $S$, $|E| = const.$ on S, and the equations  
\begin{align}%\label{eq:quasiloc. conditions}
    &R_S = \dfrac{n}{n-1}H_S^2+2|E|^2,\tag{Q1}\label{eq:quasiloc. conditions:a} \\[\medskipamount]
    & \dfrac{\partial V}{\partial \nu}=\dfrac{H_S}{n-1}V \tag{Q2}\label{eq:quasiloc. conditions:b}
\end{align}
are satisfied on $S$. 
\end{definition}

\begin{remark}
In \cite[Definition 15]{jahns2019photon} and \cite{raulot2021spinorial}, quasi-local photon spheres are defined such that $H_S$ is a positive constant. The reason is that we consider the outward pointing unit normal $\nu$, whereas in those references $\nu$ points toward the asymptotic end.
\end{remark}

\begin{definition}\label{def:quasiloc. subextremal}
A quasi-local photon sphere $S$ in an electrostatic manifold (with boundary) is called \emph{sub-extremal} if
\begin{equation}\label{eq:quasiloc. subextremal}
     H^2_S > \dfrac{n-2}{n-1}R_S.
 \end{equation}
If instead of $\,>\,$ in \eqref{eq:quasiloc. subextremal}  there is $\,=\,$ or $\,<\,$, then $S$ is called \emph{extremal} or \emph{super-extremal}, respectively.
\end{definition}

\begin{lemma}\label{lem:quasiloc}
Let $(M^n,g,V, E)$ be an electrostatic manifold with compact boundary with $\Lambda=0$ and let $S$ be a connected component of $\partial M$ such that $V|_S=const.>0$, $H_S<0$, $E\in \Gamma(N S)$ normal vector field on $S$, and $|E| = const.$ on $S$. Then $S$ is a quasi-local photon sphere. 
\end{lemma}

\begin{proof}
By~\eqref{eq:electrostatic 2 non-zero:b}, $S$ is umbilical. Lemma~\ref{lem:properties}(e) gives $H_S=\mathrm{const.}$, and~\eqref{eq:electrostatic traces:b} is exactly~\eqref{eq:quasiloc. conditions:b}.

Now we shall prove equality \eqref{eq:quasiloc. conditions:a}. Using the contracted Gauss equation and umbilicity of $S$, we obtain
         \begin{align}\label{eq:contracted Gauss}
         \begin{split}
            R_g - 2\Ric_g(\nu,\nu) &= R_S  + |B_S|^2 - H_S^2 \\
            &= R_S +\left|\dfrac{H_S}{n-1}g\right|^2 - H_S^2 \\
            &= R_S + (n-1)\dfrac{H_S^2}{(n-1)^2} - H_S^2 \\
            &= R_S - \dfrac{n-2}{n-1}H_S^2.
        \end{split}
        \end{align}
Also the formula for the Laplacian of a hypersurface implies 
        \begin{equation}\label{eq:laplacian of hypersurface}
            \Delta_g V = \Delta_S V + \Hess_gV(\nu,\nu) + H_S\dfrac{\partial V}{\partial\nu}.
        \end{equation}
Then the condition $V|_S= const.>0$, \eqref{eq:laplacian of hypersurface} and \eqref{eq:electrostatic:b} suggest 
        \begin{align}\label{Hess and H}
        \begin{split}
            \Hess_g V(\nu,\nu)&= \Delta_g V - \Delta_S V - H_S\dfrac{\partial V}{\partial\nu}  \\
            & = \dfrac{2(n-2)}{n-1} V|E|_g^2 - \dfrac{2V}{n-1}\Lambda - 0 - H_S\dfrac{\partial V}{\partial\nu}  \\
            & = \dfrac{2(n-2)}{n-1} V|E|_g^2 - \dfrac{2V}{n-1}\Lambda - \dfrac{H_S^2}{n-1}V.
        \end{split}
        \end{align}
Let us evaluate \eqref{eq:electrostatic:a} in the normal direction $\nu$ to $S$:
        \begin{align}\label{eq:Hess and Ric}
        \begin{split}
            \Hess_{g}V(\nu,\nu) &= V\Ric_g(\nu,\nu) - \dfrac{2V}{n-1}\Lambda g(\nu,\nu)  \\
            &+ 2V \left(E^\flat \otimes E^\flat\right)(\nu,\nu) - \dfrac{2}{n-1}V|E|_g^2 \, g(\nu,\nu)  \\
            &= V\Ric_g(\nu,\nu) - \dfrac{2V}{n-1}\Lambda  + 2V|E|_g^2 - \dfrac{2}{n-1}V|E|_g^2  \\
            & = V\Ric_g(\nu,\nu) - \dfrac{2V}{n-1}\Lambda + \dfrac{2(n-2)}{n-1}V|E|_g^2,
        \end{split}
        \end{align}
         where in the second line we used the condition that $E\in \Gamma(N S)$.

Combining \eqref{eq:Hess and Ric}, \eqref{Hess and H} and the fact that $V|_S=const.>0$ we obtain on $S$ 
        \begin{equation}\label{eq:Ric and H}
            \Ric_g(\nu,\nu) = -\dfrac{H_S^2}{n-1}.
        \end{equation}
  Finally, \eqref{eq:Ric and H}, \eqref{eq:contracted Gauss} and \eqref{eq:electrostatic 2 non-zero:a} with $\Lambda=0$ imply on $S$
        \begin{align}
        \begin{split}
            R_S &= \dfrac{n-2}{n-1}H_S^2 + \dfrac{2}{n-1}H_S^2 + 2|E|_g^2 + 2\Lambda  \\
            & = \dfrac{n}{n-1}H_S^2 + 2|E|_g^2 + 2\Lambda,
        \end{split}
        \end{align}
        which is exactly the equation \eqref{eq:quasiloc. conditions:a} with $\Lambda=0$. 
\end{proof}

\subsection{Asymptotically flat systems}\label{sec:asymt_flat}

Now we provide definitions that are used in both Theorem~\ref{thm:non-compact case ARNS} and Theorem~\ref{thm:non-compact_2}. We begin by defining \emph{asymptotically Reissner--Nordstr\"om systems}, followed by a definition of asymptotically flat systems. Although the former are asymptotically flat in the usual sense, we use the latter term for the weaker decay conditions discussed in~\cite[Remark~2.1]{borghini2024asymptflatelctro}. 

When working with asymptotically flat spacetime, as in \cite{kunduri2018nostaticbubbling} it is also natural to assume that the \textit{domain of outer communication} for the static spacetime defined in Section~\ref{sec:prelim} is globally hyperbolic. Then, as the authors of the aforementioned article note, topological censorship (see \cite{friedman1993topcensor, friedman1995errata, chrusciel1994topstatblackholes, galloway1995topoutercommuncation}) implies that $M$ is simply connected and, due to $d(VE^\flat)=0$ (that is, $VE^\flat$ being closed) from \eqref{eq:electrostatic:c}, we deduce that $VE^\flat$ must also be exact, and so there exists a globally defined smooth function $\Psi:M \rightarrow \mathbb{R}$ called the \emph{electric potential} such that $VE = -\nabla^g \Psi$.

However, the only assumption we require is the exactness of $V E^\flat$ and, consequently, the existence of the aforementioned electric potential $\Psi$. We make no other a priori assumptions: the domain of outer communication need not be globally hyperbolic, nor must an asymptotically flat electrostatic manifold with boundary be simply connected.

When proving uniqueness theorems, it will be more convenient for us to work with the electrostatic manifolds with boundary defined using electric potential $\Psi$. Substituting $VE = -\nabla^g\Psi$ into the defining equations~\eqref{eq:electrostatic:a}--\eqref{eq:electrostatic:d} from the introduction, we obtain the following formulation in terms of $\Psi$.
\begin{definition}
The quadruple $(M^n,g,V,\Psi)$ is an electrostatic manifold with boundary with $\Lambda=0$ if it satisfies the following equations:
\begin{align}
    &\Hess_g V = V\Ric_g + \dfrac{2}{V} d\Psi \otimes d\Psi - \dfrac{2}{(n-1)} \dfrac1V |d\Psi|^2_g \, g &&\text{in } M,\tag{PEM 1}\label{eq:potential_EM:a} \\[\medskipamount]
    &\Delta_g V = \dfrac{2(n-2)}{n-1}\dfrac{|d\Psi|_g^2}{V} &&\text{in } M,\tag{PEM 2}\label{eq:potential_EM:b} \\[\medskipamount]
    & 0 = \mydiv_g \left(\dfrac{\nabla^g \Psi}{V} \right) &&\text{in } M,\tag{PEM 3}\label{eq:potential_EM:c}
\end{align}
together with the boundary condition~\eqref{eq:electrostatic:d}. The equation $d(VE^\flat)=0$ is automatic because $d^2\Psi=0$.
\end{definition}

In this notation,~\eqref{eq:electrostatic 2 non-zero:a} becomes
\begin{equation}\tag{NPEM 1}\label{eq:potential_nonzero_EM:a}
R_g=\dfrac{2|d\Psi|_g^2}{V^2},
\end{equation}
while the boundary relation~\eqref{eq:electrostatic 2 non-zero:b} is unchanged.

For the sake of completeness, we proceed with the following definitions from \cite[Definition 2.7]{cederbaum2023equi}.

\begin{definition}\label{def:ARNS} 
A quadruple $(M^n,g,V, E)$ with $n\geqslant 3$ is an \emph{asymptotically Reissner--Nordstr\"om system of mass $m$ and charge $q$}, if $(M,g)$ is an \emph{asymptotically Reissner--Nordstr\"om manifold of mass $m$ and charge $q$}, i.e., there exists a compact set $K$ in $M$, such that $M\setminus K=\sqcup_{k=1}^KM_k$, where $M_k$ (called \textit{end}) is diffeomorphic to $\mathbb R^n\setminus \overline{\mathbb B^n}$, $\Phi_k=(x^i):M_k \rightarrow \mathbb R^n\setminus \overline{\mathbb B^n}$, and if, as $|x| \rightarrow \infty$, the following equations are satisfied

\begin{align*}
    (\Phi_{k_*}g)_{ij}-(g_{m,q})_{ij}&=O_2(|x|^{-(n-1)}),  \\[\medskipamount]
    \Phi_{k_*}V-V_{m,q}&=O_2(|x|^{-(n-1)}),\\[\medskipamount]
    \Phi_{k_*}\Psi-\Psi_{m,q}&=O_2(|x|^{-(n-1)}),
\end{align*}

\noindent where $|\cdot|$ denotes the Euclidean norm and $(g_{m,q})_{ij}, V_{m,q}, \Psi_{m,q}$ are all in isotropic coordinates.
\end{definition}

As in \cite{borghini2024asymptflatelctro} we define 3-dimensional asymptotically flat electrostatic systems with weaker asymptotic decay.
\begin{definition}\label{def:AF}
    An electrostatic manifold with compact connected boundary $(M^3, g, V, \Psi)$ with nonvanishing $V$ is asymptotically flat if there exists a compact set $K$ in $M$ with $\partial M \subset K$ such that $M \setminus K$ is diffeomorphic to $\mathbb{R}^3 \setminus \overline{\mathbb{B}^3}$, $\Phi_k=(x^i):M_k \rightarrow \mathbb R^n\setminus \overline{\mathbb B^3}$ and, if, $|x|_\delta \rightarrow \infty$, the following equations are satisfied
    \begin{align*}
        &g_{ij} = (\delta_{ij} + \eta_{ij}) \, dx^i \otimes dx^j, \quad \eta_{ij} = O_1(1), \\[\medskipamount]
        & V = 1 - \dfrac{\mu}{|x|} + O_2(|x|^{-1}), \quad \Psi = O_1(1),
    \end{align*}
    \noindent where $\mu \in \mathbb{R}$ is a constant called the \emph{mass}.
    And for the case where $\mu=0$, the following equation is additionally satisfied
    \begin{align*}
        \Psi = \dfrac{\kappa}{|x|}+ O_2(|x|^{-1}),
    \end{align*}
    \noindent where $\kappa \in \mathbb{R}$ is a constant called the \emph{charge}.

\end{definition}

\begin{remark}
Analogously to \cite{cederbaum2023equi, borghini2024asymptflatelctro}, for both definitions we say a smooth function $f: U \rightarrow \mathbb{R}$ on some open subset of $\mathbb{R}^n$ satisfies $f = O_l(|x|^p)$ with $l \in \mathbb{N}, p \in \mathbb{R}$ if there exists a constant $C>0$ such that 
\begin{align*}
    \sum_{|J|\leqslant l} |\partial^J f| \leqslant C|x|^{-p-|J|}
\end{align*}
\noindent as $|x| \rightarrow \infty$ for every multi-index $J$ such that $|J|\leqslant l$ $\Big($here $\displaystyle \partial^J f = \dfrac{\partial^{|J|}f}{\partial x_1^{j_1}...\partial x_n^{j_n}}$$\Big)$.
\end{remark}

\begin{remark}
As in \cite[Remark 2.8]{cederbaum2023equi} and \cite[Remark 2.2]{borghini2024asymptflatelctro} we note that asymptotically flat (and, hence, asymptotically Reissner--Nordstr\"om) systems (with boundary) are necessarily metrically complete and geodesically complete (up to the boundary) with at most finitely many boundary components which are necessarily all closed.
\end{remark}

\subsection{Proofs of the uniqueness theorems for non-compact electrostatic manifolds with boundary}

We use the Reissner--Nordstr\"om regions $RN^n_-$ and $RN^n_+$ introduced above, in analogy with~\cite[Example~6]{medvedev2024static}. For the super-extremal cases, the photon-sphere classification is given in Appendix~\ref{ap:rps}. When a photon sphere exists, $RN^n_-$ denotes the exterior of the outermost one. Otherwise, let $r_0$ be the lower endpoint of the maximal radial interval on which $V_{m,q}$ is defined, and take $RN^n_-=[r_+,\infty)\times\mathbb S^{n-1}$ with the model data, where $r_+>r_0$. The spherically symmetric equipotential photon surface through $\{r=r_+\}$ is described in~\cite[Section~2]{borghini2024asymptflatelctro}.

Now we are ready to prove the uniqueness theorems.

\begin{proof}[Proof of Theorem~\ref{thm:non-compact case ARNS}]\hfill
\begin{center}
    I.
\end{center}
    
Since $V$ is nowhere zero and tends to $1$ at the asymptotic end, it is positive. Lemma~\ref{lem:quasiloc} identifies every boundary component as a quasi-local photon sphere. The sub-extremality assumption then allows us to apply~\cite[Theorem~3]{jahns2019photon}.

\begin{center}
    II.
\end{center}
    
\emph{Step 1: rigidity of the exterior region.}
Write $S=\partial M$. Since $\Delta_gV=|E|_g^2V$, the maximum principle excludes a bounded nodal domain whose boundary is contained entirely in $\Sigma$. The connectedness of $S$, the one-ended assumption, and the separation hypothesis therefore give two regions after cutting along $\Sigma$: the noncompact region $N$ containing the asymptotic end and a compact region $\Omega$ with $\partial\Omega=\Sigma\sqcup S$. We include their boundaries in this notation. We have $V>0$ in the interior of $N$ and, by nondegeneracy of $\Sigma$, $V<0$ on $\Omega\setminus\Sigma$.

By Lemma~\ref{lem:properties}, $\partial N=\Sigma$ is a nondegenerate static horizon. Hence~\cite[Corollary~4]{jahns2019photon} identifies $N$ with the full Reissner--Nordstr\"om exterior $[r_h,\infty)\times\mathbb S^2$, where $m>|q|$ and $r_h=m+\sqrt{m^2-q^2}$. In particular, $\Sigma$ is connected and round. With the asymptotic normalization in the theorem, the fields on $N$ are
\[
V=\sqrt{1-\frac{2m}{r}+\frac{q^2}{r^2}},\qquad \Psi=\frac qr,\qquad E=\frac{q}{r^2}\sqrt{1-\frac{2m}{r}+\frac{q^2}{r^2}}\,\partial_r.
\]
Here $N$ includes the horizon-to-photon-sphere region, whereas $RN^3_-$ starts at the photon sphere.

\emph{Step 2: matching an exterior to the Robin boundary.}
Let $\nu$ point out of $\Omega$ along $S$, and put
\[
v=V|_S<0,\qquad H=H_g|_S>0,\qquad e=g(E,\nu)|_S.
\]
These quantities are constant: this follows for $H$ from Lemma~\ref{lem:properties}(e), and for $e$ from the normality of $E$, the constancy of $|E|$, and the connectedness of $S$. The Gauss-equation calculation in the proof of Lemma~\ref{lem:quasiloc} uses only $V|_S\neq0$ and gives
\[
R_S=\frac32H^2+2e^2>0.
\]
Thus $S$ is a round sphere, since it is closed and orientable with constant positive Gaussian curvature. Write its metric as $\gamma=R^2g_{\mathbb S^2}$. Then
\begin{equation}\label{eq:compact-boundary-data}
1=\frac34H^2R^2+e^2R^2,\qquad B=\frac H2\gamma,\qquad \partial_\nu V=\frac H2v.
\end{equation}

Define parameters for the exterior to be attached by
\begin{equation}\label{eq:compact-cap-parameters}
Q=eR^2,\qquad \mu=\frac{R^2+2Q^2}{3R},\qquad
f(r)=\sqrt{1-\frac{2\mu}{r}+\frac{Q^2}{r^2}}.
\end{equation}
The sub-extremality inequality is $H^2>1/R^2$. Together with~\eqref{eq:compact-boundary-data}, it implies $Q^2<R^2/4$ and hence
\[
\mu^2-Q^2=\frac{(R^2-Q^2)(R^2-4Q^2)}{9R^2}>0.
\]
Moreover,
\begin{equation}\label{eq:compact-cap-matching}
R^2-3\mu R+2Q^2=0,\qquad f(R)=\frac{HR}{2},\qquad f(R)f'(R)=\frac{f(R)^2}{R}.
\end{equation}
Here $f'=df/dr$. Since $Q^2<R^2/4$, $R$ is the outer photon-sphere radius of this sub-extremal model. 

Equip $[R,\infty)\times\mathbb S^2$ with the Reissner--Nordstr\"om metric $\widetilde g=f(r)^{-2}dr^2+r^2g_{\mathbb S^2}$ and glue it to $\Omega$ by identifying $\{R\}\times\mathbb S^2$ with $S$ through a round isometry. The metrics induced on these identified boundaries agree with $\gamma$. With respect to the common unit normal pointing from $\Omega$ into the attached exterior, given there by $\nu=f\partial_r$, the second fundamental forms also agree: $f(R)\gamma/R=H\gamma/2$. Hence the metric and its first normal
derivative match across $S$ in Gaussian coordinates.

It remains to match the fields, including their normalization and signs. The function $\Psi$ is constant on $S$, because $d\Psi=-VE^\flat$ and $E$ is normal there; denote this constant by $\psi_S$. Set
\[
c=-\frac{v}{f(R)}>0,\qquad b=\psi_S+\frac{cQ}{R}.
\]
Retain $(g,V,\Psi,E)$ on $\Omega$ and define on the attached exterior
\begin{equation}\label{eq:compact-cap-fields}
\widetilde V=-cf,\qquad\widetilde\Psi=b-\frac{cQ}{r},\qquad \widetilde E=\frac{Q}{r^2}f\partial_r.
\end{equation}
Then $\widetilde V\widetilde E^\flat=-d\widetilde\Psi$, and on $S$
\begin{align*}
\widetilde V&=v,&\partial_\nu\widetilde V&=-cf(R)f'(R)=\frac H2v,\\ \widetilde\Psi&=\psi_S,& \partial_\nu\widetilde\Psi&=\frac{cQf(R)}{R^2}=-ve,\qquad \widetilde E=e\nu.
\end{align*}
These are exactly the values and normal derivatives on the $\Omega$ side; the tangential derivatives of both potentials vanish on $S$. Consequently, $\widetilde g$, $\widetilde V$, and $\widetilde\Psi$ are $C^{1,1}$ across $S$ and smooth on either side. The electric field is continuous there, and its normal flux agrees. The electrostatic equations hold on both sides and weakly across $S$, with no surface source introduced by the gluing.

\emph{Step 3: rigidity of the compact region.}
On the resulting one-ended manifold $\widetilde M$, normalize the fields by
\[
U=-\frac{\widetilde V}{c},\qquad
\Phi=\frac{b-\widetilde\Psi}{c}.
\]
Then $U>0$ away from the horizon $\Sigma$, $U\widetilde E=-\nabla^{\widetilde g}\Phi$, and the attached end has exactly the normalized Reissner--Nordstr\"om data $(f,Q/r)$. Thus the horizon rigidity argument must be applied to $(\widetilde g,U,\Phi)$, with the constants $c$ and $b$ fixed above.

We spell out why the gluing has the regularity needed for that argument. Double $\widetilde M$ across $\Sigma$, extending the metric and $\Phi$ evenly and $U$ oddly; write the resulting fields as $(\widehat g,\widehat U,\widehat\Phi)$. On $\Sigma$, the second fundamental form vanishes, $d\Phi=-U\widetilde E^\flat=0$, and $\Hess_{\widetilde g} U=0$. In Gaussian coordinates these identities give matching first derivatives, so the doubled fields are $C^{1,1}$. On the positive half, the functions $W_\pm=U-1\pm\Phi$ satisfy
\[
\Delta_{\widetilde g}W_\pm=\pm\frac{\langle\nabla^{\widetilde g}\Phi,\nabla^{\widetilde g} W_\pm\rangle_{\widetilde g}}{U}.
\]
See~ \cite[Lemma~24, equation~(22)]{jahns2019photon}. The righthand side extends to the horizon since $\nabla^{\widetilde g}\Phi/U=-\widetilde E$. Moreover, $W_\pm<0$ far out because $\mu>|Q|$, while their derivative into the positive half at $\Sigma$ is $\partial_{\eta_+}U>0$, where $\eta_+$ points into that half. The maximum principle, including the transmission surface $S$, therefore gives $W_\pm<0$ throughout the positive half. It follows, also on the reflected half, that
\[
\Theta=\frac{(1+\widehat U)^2-\widehat\Phi^2}{4}>0.
\]

The conformal horizon-uniqueness argument of~\cite[proof of Theorem~3]{jahns2019photon}, based on~\cite{kunduri2018nostaticbubbling}, now applies to $\Theta^2\widehat g$. It has nonnegative scalar curvature, its positive end has zero ADM mass, and its negative end admits a one-point compactification. The rigidity case of the positive mass theorem~\cite{bartnik1986mass}, in the $C^{1,1}$ form used in that argument, makes the compactified metric Euclidean. The corresponding recovery of the conformal factor and the potentials identifies $(\widetilde M,\widetilde g,U,\Phi)$ with the normalized Reissner--Nordstr\"om exterior of parameters $(\mu,Q)$. In particular, this proves rigidity of the whole compact region:
\begin{equation}\label{eq:compact-region-rigidity}
\Omega\cong[r_h(\mu,Q),R]\times\mathbb S^2,\qquad V=-cf,\qquad \Psi=b-\frac{cQ}{r},\qquad E=\frac{Q}{r^2}f\partial_r,
\end{equation}
where $r_h(\mu,Q)=\mu+\sqrt{\mu^2-Q^2}$ is the outer horizon radius of this model, and $r$ increases from $\Sigma$ toward $S$. 

\emph{Step 4: compatibility across the original zero set.}
The two regions identified above initially have parameters $(m,q)$ and $(\mu,Q)$. The metric induced on their common horizon $\Sigma$ gives
\[
r_h(\mu,Q)=r_h(m,q)=r_h,\qquad r_h(m,q)=m+\sqrt{m^2-q^2}.
\]
Let $\eta$ be the unit normal to $\Sigma$ pointing into $N$; this is the outward normal of $\Omega$ along $\Sigma$. By Steps~1 and~2,
\[
g(E,\eta)|_\Sigma=\frac{q}{r_h^2},\qquad g(E,\nu)|_S=\frac{Q}{R^2}.
\]
Since $\operatorname{div}_g E=0$, the divergence theorem gives
\[
0=\int_\Omega \operatorname{div}_g E\,dv_g=\int_S g(E,\nu)\,ds_g+\int_\Sigma g(E,\eta)\,ds_g=4\pi(Q+q).
\]
Hence $Q=-q$, and therefore
\[
\mu=\frac{r_h^2+Q^2}{2r_h}
=\frac{r_h^2+q^2}{2r_h}=m.
\]
The original smooth potential has a single nonzero surface gravity on $\Sigma$. Its values computed from the two regions are
\[
\kappa=\frac{r_h^2-q^2}{2r_h^3}
=c\frac{r_h^2-Q^2}{2r_h^3}.
\]
Hence $c=1$. Finally, continuity of $\Psi$ at $\Sigma$, together with its normalization $\Psi=q/r$ on $N$, implies $b=0$.

To verify smooth compatibility across $\Sigma$, use signed arclength $s$ from $\Sigma$, positive on $N$ and negative on $\Omega$. Thus $ds=dr/f$ on $N$ and $ds=-dr/f$ on $\Omega$, since the radial coordinate on each side increases away from the horizon. After composing one of the round identifications with a spherical isometry, the angular coordinates agree on $\Sigma$; extend them along the normal radial curves. Both metrics then take the form $g=ds^2+r(s)^2g_{\mathbb S^2}$. Writing primes for derivatives with respect to $s$, we have $r'=f(r)$ on $N$ and $r'=-f(r)$ on $\Omega$. Hence, on either side,
\[
r''=f(r)\frac{df}{dr}=\frac12\frac{d}{dr}\left(1-\frac{2m}{r}+\frac{q^2}{r^2}\right)=\frac{m}{r^2}-\frac{q^2}{r^3},\qquad r(0)=r_h,\qquad r'(0)=0.
\]
This second-order equation is regular at $r=r_h>0$, so uniqueness for its initial-value problem identifies the two radial profiles with a single smooth solution across $s=0$. This solution is even, since $r(-s)$ satisfies the same equation and initial conditions. Thus the metrics fit smoothly across $\Sigma$. The fields on both sides become
\[
V=\frac{dr}{ds},\qquad E=\frac{q}{r(s)^2}\partial_s,\qquad
\Psi=\frac{q}{r(s)}.
\]
Thus $V$ changes sign smoothly while $E$ extends smoothly through the horizon. Since $R$ is the outer photon-sphere radius, the full manifold is $[-s_{\mathrm{ph}},\infty)\times\mathbb S^2$ with these data, namely $RN^3_+$. This proves the asserted global rigidity of the metric, the static potential, and the electric field.
\end{proof}

\begin{proof}[Proof of Theorem~\ref{thm:non-compact_2}]\hfill

As in part~I of the preceding proof, $V>0$, and Lemma~\ref{lem:quasiloc} identifies the boundary as a quasi-local photon sphere. The conclusion follows by applying the electrostatic equipotential photon surface uniqueness theorem~\cite[Theorem~3.2]{borghini2024asymptflatelctro} under the stated boundary hypotheses.
\end{proof}

\section{Robin electrostatic quadruples}\label{sec:robin-electrostatic-quadruples}

We begin this section with the proof of Theorem~\ref{prop:electrostatic-criticality}. 

\begin{proof}[Proof of Theorem~\ref{prop:electrostatic-criticality}]
Put
\[
q=\dfrac{2(n-2)}{n-1}|E|_g^2,\qquad\beta=\dfrac{H_g}{n-1}\ \text{on }S,\qquad\lambda_0=\dfrac{2\Lambda}{n-1}.
\]
Equations~\eqref{eq:electrostatic-trace} and \eqref{eq:electrostatic-robin} give
\begin{equation}\label{eq:potential-eigenproblem}
-\Delta_gV+qV=\lambda_0V\quad\text{in }\Int(M),\qquad V=0\quad\text{on }D,\qquad\partial_\nu V=\beta V\quad\text{on }S.
\end{equation}

The local zero-set arguments underlying Lemma~\ref{lem:properties}(a)--(c) apply to the Dirichlet face $D$: they use only the interior Hessian and Maxwell equations. Thus $D$ is totally geodesic, each $\kappa_j$ is a positive constant, and $E$ is normal to $D$. At a common edge, the Robin condition gives $g(\nabla^gV,\nu_S)=0$; since $\nabla^gV$ is a nonzero normal to $D$, the faces meet orthogonally.

Set
\[
I=\int_MV^2\,dv_g,\qquad u=\dfrac{V}{\sqrt I}.
\]
The positive interior function $u$ solves \eqref{eq:potential-eigenproblem} and is normalized in $L^2(M,dv_g)$. A first eigenfunction of the mixed problem can be chosen positive in the interior. Orthogonality of eigenfunctions for distinct eigenvalues therefore gives $\mathcal F(g,E)=\lambda_0$. This first eigenvalue is simple.

For a smooth variation $(g_t,E_t)$, write
\[
q_t=\dfrac{2(n-2)}{n-1}|E_t|_{g_t}^2,\qquad\beta_t=\dfrac{H_{g_t}}{n-1},\qquad h=\dot g_0.
\]
Differentiating the quadratic forms on the fixed space $H_D^1(M,dv_g)$ gives
\begin{align}
\mathcal F'(0)={}&-\int_M\langle du\otimes du,h\rangle_g\,dv_g\nonumber\\&+\dfrac12\int_M\left(|\nabla^gu|_g^2+(q-\lambda_0)u^2\right)\tr_g h\,dv_g+\int_Mu^2\dot q_0\,dv_g\nonumber\\&-\int_Su^2\dot\beta_0\,ds_g-\dfrac12\int_S\beta u^2\tr_S h\,ds_g.
\label{eq:general-metric-variation}
\end{align}
To prove this, put $\lambda_t=\mathcal F(g_t,E_t)$. The boundary faces $D$ and $S$ remain fixed. Since the first eigenvalue is simple, we may choose a differentiable family of first eigenfunctions $u_t\in H_D^1(M,dv_g)$ satisfying
\[
\int_Mu_t^2\,dv_{g_t}=1.
\]
Write
\[
u=u_0,\qquad \dot u=\left.\dfrac{d}{dt}\right|_{t=0}u_t.
\]
Because $D$ is fixed, we have $\dot u\in H_D^1(M, dv_g)$.

The variation formulas for the volume and boundary measures are
\[
\left.\dfrac{d}{dt}\right|_{t=0}dv_{g_t}=\dfrac12\operatorname{tr}_g h\,dv_g,\qquad\left.\dfrac{d}{dt}\right|_{t=0}ds_{g_t}=\dfrac12\operatorname{tr}_S h\,ds_g.
\]
Here $\operatorname{tr}_S h$ is the trace of the restriction of $h$ to $TS$ with respect to the induced metric.

Differentiating the identity $g_t^{ik}(g_t)_{kj}=\delta^i_j$ gives
\[
\left.\dfrac{d}{dt}\right|_{t=0}g_t^{ij}=-g^{ik}g^{j\ell}h_{k\ell}=-h^{ij}.
\]
Consequently,
\begin{align*}
\left.\dfrac{d}{dt}\right|_{t=0}|\nabla^{g_t}u_t|_{g_t}^2&=\left.\dfrac{d}{dt}\right|_{t=0}\left(g_t^{ij}\partial_i u_t\,\partial_j u_t\right)=-h^{ij}\partial_i u\,\partial_j u+2g^{ij}\partial_i u\,\partial_j\dot u\\&=-\langle du\otimes du,h\rangle_g+2\langle\nabla^gu,\nabla^g\dot u\rangle_g.
\end{align*}

Differentiating the normalization yields
\[
0=2\int_Mu\dot u\,dv_g+\dfrac12\int_Mu^2\operatorname{tr}_g h\,dv_g.
\]

Since $u_t$ is normalized, its eigenvalue equals its energy:
\[
\lambda_t=\int_M\left(|\nabla^{g_t}u_t|_{g_t}^2+q_tu_t^2\right)dv_{g_t}-\int_S\beta_tu_t^2\,ds_{g_t}.
\]
Using the preceding variation formulas, we obtain
\begin{align*}
\lambda'_0={}&-\int_M\langle du\otimes du,h\rangle_g\,dv_g+\dfrac12\int_M\left(|\nabla^gu|_g^2+qu^2\right)\operatorname{tr}_g h\,dv_g+\int_Mu^2\dot q_0\,dv_g\\&-\int_Su^2\dot\beta_0\,ds_g-\dfrac12\int_S\beta u^2\operatorname{tr}_S h\,ds_g\\&+2\left(\int_M\langle\nabla^gu,\nabla^g\dot u\rangle_g\,dv_g+\int_Mqu\dot u\,dv_g-\int_S\beta u\dot u\,ds_g\right).
\end{align*}

For every $w\in H_D^1(M,dv_g)$, the weak eigenvalue equation at $t=0$ reads
\[
\int_M\langle\nabla^gu,\nabla^gw\rangle_g\,dv_g+\int_Mquw\,dv_g-\int_S\beta uw\,ds_g=\lambda_0\int_Muw\,dv_g.
\]
Taking $w=\dot u$, we see that the last line of the energy variation is
\begin{align*}
&2\left(\int_M\langle\nabla^gu,\nabla^g\dot u\rangle_g\,dv_g+\int_Mqu\dot u\,dv_g-\int_S\beta u\dot u\,ds_g\right)\\&\qquad=2\lambda_0\int_Mu\dot u\,dv_g=-\dfrac{\lambda_0}{2}
\int_Mu^2\operatorname{tr}_g h\,dv_g.
\end{align*}
Substituting this identity proves~\eqref{eq:general-metric-variation}.

This calculation uses the weak eigenvalue equation on the fixed space $H_D^1(M,dv_g)$, so it remains valid when $D$ and $S$ meet along an edge. If $S=\varnothing$, the integrals over $S$ are omitted.

Introduce
\[
T=\left(|\nabla^gu|_g^2+(q-\lambda_0)u^2\right)g-2\,du\otimes du.
\]
Since $\Delta_gu=(q-\lambda_0)u$, direct differentiation gives
\begin{equation}\label{eq:stress-divergence}
\mydiv_gT=u^2\,dq.
\end{equation}
For $h=\mathcal L_Xg$, integration by parts yields
\begin{equation}\label{eq:stress-integration}
\dfrac12\int_M\langle T,\mathcal L_Xg\rangle_g\,dv_g=\int_{\partial M}T(X,\nu)\,ds_g-\int_Mu^2X(q)\,dv_g.
\end{equation}
On $D$, let $f=g(X,\nu)$. Since $u=0$ there,
\[
T(X,\nu)=-f(\partial_\nu u)^2.
\]
On $S$, tangency of $X$ and the Robin condition give
\[
T(X,\nu)=-2\beta uX(u),\qquad
\dfrac12\tr_S(\mathcal L_Xg)=\mydiv_SX.
\]
Hence
\[
-2\beta uX(u)-\beta u^2\mydiv_SX=-\mydiv_S(\beta u^2X)+u^2X(\beta).
\]
The integral of this divergence vanishes. Indeed, a closed Robin face has no boundary, while on an edge shared with $D$ the corresponding boundary integrand vanishes because $u=0$. Combining these identities, we obtain
\begin{align}
\mathcal F'(0)={}&-\int_D f(\partial_\nu u)^2\,ds_g\nonumber\\&+\int_Mu^2\bigl(\dot q_0-X(q)\bigr)\,dv_g+\int_Su^2\bigl(X(\beta)-\dot\beta_0\bigr)\,ds_g.
\label{eq:transport-correction}
\end{align}

For the variations under consideration,
\begin{align*}
\left.\dfrac{d}{dt}\right|_{t=0}|E_t|_{g_t}^2&=(\mathcal L_Xg)(E,E)+2g([X,E],E)=X(|E|_g^2).
\end{align*}
Thus $\dot q_0=X(q)$. Moreover, since $X$ is tangent to $S$, we have $\dot\beta_0=X(\beta)$. This is a local first-variation identity on the relative interior of $S$, obtained from local changes of coordinates preserving $S$, and it extends to the edge by continuity. It does not
require a flow preserving $D$. We conclude that

\begin{equation}\label{eq:electrostatic-hadamard}
\mathcal F'(0)=-\int_D f(\partial_\nu u)^2\,ds_g=-\dfrac1I\sum_{j=1}^{\ell}\kappa_j^2\int_{\Sigma_j}f\,ds_g.
\end{equation}

The volume variation is
\begin{align}
\left.\dfrac{d}{dt}\right|_{t=0}\Vol_{g_t}(M)&=\dfrac12\int_M\tr_g(\mathcal L_Xg)\,dv_g\nonumber\\&=\int_{\partial M}g(X,\nu)\,ds_g=\int_Df\,ds_g,
\label{eq:volume-variation}
\end{align}
where tangency to $S$ was used in the last equality. If all $\kappa_j$ agree, equations \eqref{eq:electrostatic-hadamard} and \eqref{eq:volume-variation} give stationarity for every admissible variation.

Conversely, suppose stationarity holds. For $i\neq j$, choose a smooth vector field $X$, supported away from $S$ and the edges, whose normal component satisfies
\[
\int_{\Sigma_i}f\,d\sigma_g=1,\qquad\int_{\Sigma_j}f\,d\sigma_g=-1,\qquad f=0\quad\text{on all other components}.
\]
Such fields are constructed in collars of compact patches in the relative interiors of the two Dirichlet faces. Set
\[
h=\mathcal L_Xg,\qquad \widehat g_t=g+th,
\]
and, for sufficiently small $|t|$, define
\[
g_t=\left(\dfrac{\Vol_g(M)}{\Vol_{\widehat g_t}(M)}\right)^{2/n}
\widehat g_t,\qquad E_t=E+t[X,E].
\]
Equation~\eqref{eq:volume-variation} shows that the scaling factor has zero derivative at $t=0$. Consequently, $(g_t,E_t)$ is admissible and has the required initial derivatives. Stationarity and \eqref{eq:electrostatic-hadamard} imply $\kappa_i^2=\kappa_j^2$. Since both constants are positive, they are equal.
\end{proof}

\begin{remark}
The functional in \eqref{eq:spectral-functional} recomputes both $|\widetilde E|_{\widetilde g}^2$ and $H_{\widetilde g}$. The simultaneous first variations $\dot g_0=\mathcal L_Xg$ and $\dot E_0=\mathcal L_XE$ are the intrinsic counterparts of pulling back fixed ambient metric and electric-field data under a domain deformation. Formula~\eqref{eq:transport-correction} records the extra terms that arise with other choices of variation.

In the vacuum case,~\eqref{eq:electrostatic 2 non-zero:a} gives $R_g=2\Lambda$, and the local boundary identity in Lemma~\ref{lem:properties}(e) shows that $H_g$ is constant on each connected Robin face. Thus $\dot q_0=\dot\beta_0=0$ for admissible variations, and the first-variation statement reduces to the static proposition with fixed Robin coefficient. The two functionals need only agree to first order at the base point.
\end{remark}

As a corollary, we obtain. 

\begin{corollary}\label{cor:pure-dirichlet-electrostatic}
Let $(M^n,g,V,E)$ be a Robin electrostatic quadruple. Suppose that $S=\varnothing$ and $D=\partial M\neq\varnothing$. Then the first Dirichlet eigenvalue of $-\Delta_g+\dfrac{2(n-2)}{n-1}|E|_g^2$ satisfies
\begin{equation}\label{eq:dirichlet-electrostatic-eigenvalue}
\lambda_1^D\left(-\Delta_g+\dfrac{2(n-2)}{n-1}|E|_g^2\right)=\dfrac{2\Lambda}{n-1}>0.
\end{equation}
In particular, $\Lambda>0$, so no compact quadruple of this type exists for $\Lambda\leqslant0$.

More precisely, if $\lambda_1^D(M,g)$ denotes the first Dirichlet eigenvalue of $-\Delta_g$, then
\begin{equation}\label{eq:dirichlet-lambda-bound}
\dfrac{2\Lambda}{n-1}\geqslant\lambda_1^D(M,g)+\dfrac{2(n-2)}{n-1}\min_M|E|_g^2,
\end{equation}
and hence $\Lambda>(n-2)\min_M|E|_g^2$.

Let $\mathcal F_D$ denote the functional~\eqref{eq:spectral-functional} with $S=\varnothing$.
It is stationary under every volume-preserving variation with initial derivatives $\dot g_0=\mathcal L_Xg$ and $\dot E_0=[X,E]$, with no restriction on $X$ at the boundary, if and only if the constants $|\nabla^gV|_g$ agree on all boundary components.
\end{corollary}

\begin{proof}
The eigenvalue identity and criticality assertion follow from Theorem~\ref{prop:electrostatic-criticality} with all
$S$-integrals omitted. Multiplying \eqref{eq:electrostatic-trace} by $V$ and integrating gives
\begin{equation}\label{eq:dirichlet-energy}
\dfrac{2\Lambda}{n-1}\int_MV^2\,dv_g=\int_M|\nabla^gV|_g^2\,dv_g+\dfrac{2(n-2)}{n-1}\int_M|E|_g^2V^2\,dv_g.
\end{equation}
The right-hand side is strictly positive, because $V$ is nonzero and vanishes on the nonempty boundary. This proves $\Lambda>0$. The variational characterization of the first Dirichlet eigenvalue, or the Poincar\'e inequality applied to \eqref{eq:dirichlet-energy}, proves \eqref{eq:dirichlet-lambda-bound}. Finally, $\lambda_1^D(M,g)>0$ gives the stated strict inequality.
\end{proof}

\begin{remark}
When $|E|_g^2=e^2$ is constant, one may rewrite
\eqref{eq:dirichlet-electrostatic-eigenvalue} as
\[
\lambda_1^D(M,g)=\dfrac{2}{n-1}\bigl(\Lambda-(n-2)e^2\bigr).
\]
In the vacuum case this becomes $\lambda_1^D(M,g)=R_g/(n-1)$.
\end{remark}

The following examples illustrate the preceding theorems.

\begin{example}\label{ex:round-hemisphere}
Let
\[
M=\mathbb S^n_+=\{x\in\mathbb S^n:x_{n+1}\geqslant0\},\qquad g=g_{\mathbb S^n},\qquad V=x_{n+1},\qquad E=0.
\]
Take $D=\partial M$, $S=\varnothing$, and $\Lambda=n(n-1)/2$. The identities
\[
\Hess_gV=-Vg,\qquad \Delta_gV=-nV,\qquad \Ric_g=(n-1)g
\]
verify the defining equations. Moreover, $|\nabla^gV|_g=1$ on $D$. Thus this is a purely Dirichlet critical example with connected boundary and $\mathcal F_D(g,0)=n$.
\end{example}

\begin{example}\label{ex:charged-products}
Fix $e\neq0$ and $\Lambda>(n-2)e^2$, and set
\[
\omega^2=\dfrac{2}{n-1}\bigl(\Lambda-(n-2)e^2\bigr),\qquad \rho^2=\dfrac{(n-1)(n-2)}{2(\Lambda+e^2)}.
\]
Let $h=\rho^2g_{\mathbb S^{n-1}}$, and consider
\[
g=dt^2+h,\qquad V(t,x)=\sin(\omega t),\qquad E=e\,\partial_t.
\]
Since
\[
\Ric_h=\dfrac{2(\Lambda+e^2)}{n-1}h,\qquad\Hess_gV=-\omega^2V\,dt^2,\qquad\Delta_gV=-\omega^2V,
\]
the two electrostatic equations hold. In addition,
\[
\mydiv_gE=0,\qquad d(VE^\flat)=d\bigl(e\sin(\omega t)\,dt\bigr)=0.
\]
These formulas give the following three examples.

\emph{(i) Empty Robin face.}
On
\[
M=[0,\pi/\omega]\times\mathbb S^{n-1},\qquad D=\partial M,\qquad S=\varnothing,
\]
the two Dirichlet components have $\kappa_1=\kappa_2=\omega$. Consequently this is a charged critical example with disconnected Dirichlet boundary. Its spectral value is
\[
\omega^2+\dfrac{2(n-2)}{n-1}e^2=\dfrac{2\Lambda}{n-1}.
\]

\emph{(ii) Connected Dirichlet face and a Robin face.}
Restrict instead to
\[
M=[0,\pi/(2\omega)]\times\mathbb S^{n-1},\qquad D=\{0\}\times\mathbb S^{n-1},\qquad S=\{\pi/(2\omega)\}\times\mathbb S^{n-1}.
\]
The product slice $S$ is totally geodesic and $\partial_\nu V=0$ there. Hence it satisfies the full Robin boundary equation, whose trace is the Neumann condition. Here $V=1$ on $S$, $\kappa_1=\omega$, and the quadruple is critical.

\emph{(iii) Dirichlet and Robin faces meeting along an edge.}
In (i), replace $\mathbb S^{n-1}$ by its closed hemisphere $\mathbb S^{n-1}_+$. Set
\[
D=\bigl(\{0\}\cup\{\pi/\omega\}\bigr) \times\mathbb S^{n-1}_+,\qquad S=[0,\pi/\omega]\times\partial\mathbb S^{n-1}_+.
\]
The lateral face $S$ is totally geodesic, and $\partial_\nu V=0$ on $S$. Its intersections with the two Dirichlet faces are orthogonal. Again $\kappa_1=\kappa_2=\omega$, so the quadruple is critical.
\end{example}

For the remaining examples, it is useful to record the radial Einstein--Maxwell family. Given constants $m,Q,\Lambda$, let
\begin{equation}\label{eq:radial-family}
\begin{gathered}
F(r)=1-\dfrac{2m}{r^{n-2}}+\dfrac{Q^2}{r^{2(n-2)}}-\dfrac{2\Lambda r^2}{n(n-1)},\\
g=\dfrac{dr^2}{F(r)}+r^2g_{\mathbb S^{n-1}},\qquad
V=\sqrt{F(r)},\qquad E=\sqrt{\dfrac{(n-1)(n-2)}2}\,
\dfrac{Q}{r^{n-1}}\sqrt{F(r)}\,\partial_r.
\end{gathered}
\end{equation}
For $\Lambda=0$ and $Q=q$, these are the Reissner--Nordstr\"om data~\eqref{eq:rn-data}. We use the family on intervals where $F>0$. Appendix~\ref{ap:rps} verifies the electrostatic equations, describes the smooth completion at simple zeros of $F$, and determines the Robin spheres for arbitrary $\Lambda$. In particular, it gives the surface-gravity formula~\eqref{eq:radial-surface-gravity} used below. In the following examples, $F'$ always means $dF/dr$; derivatives of the warping function $r(s)$ in the appendix are taken with respect to arclength $s$.

\begin{example}[A horizon-to-photon-sphere Reissner--Nordstr\"om region]
\label{ex:rn-robin}
In~\eqref{eq:radial-family}, take $\Lambda=0$ and $m>|Q|>0$. With $q=Q$, let $r_h$ be given by~\eqref{eq:rn-horizon} and put $r_p=r_{\mathrm{ph}}^+$ as in~\eqref{eq:photon_sphere:a}. Let $M$ be the smooth metric completion of $(r_h,r_p]\times\mathbb S^{n-1}$, and set
\[
D=\{r_h\}\times\mathbb S^{n-1},\qquad S=\{r_p\}\times\mathbb S^{n-1}.
\]
Here $r_p>r_h$, $V>0$ away from $D$, and the completion is compact. Appendix~\ref{ap:rps} verifies the full Robin boundary condition on $S$. With the outward normal pointing toward increasing $r$, it also gives $H_g|_S>0$, while~\eqref{eq:radial-surface-gravity} gives $\kappa_1>0$. Thus $(M,g,V,E)$ is a Robin electrostatic quadruple.
Since $D$ is connected, the quadruple is critical and $\mathcal F(g,E)=0$. This example also shows why the conclusion $\Lambda>0$ in Corollary~\ref{cor:pure-dirichlet-electrostatic} depends on the absence of the Robin face.
\end{example}

\begin{example}[A charged purely Dirichlet example with equal surface gravities]
\label{ex:lukewarm}
In dimension $n=3$, fix $a>0$ and $m>0$ with $4am<1$. Take $Q=m$, $\Lambda=3a^2$, and
\[
F(r)=\left(1-\dfrac mr\right)^2-a^2r^2.
\]
Define
\[
\delta=\sqrt{1-4am},\qquad r_b=\dfrac{1-\delta}{2a},\qquad r_c=\dfrac{1+\delta}{2a}.
\]
Then $F>0$ on $(r_b,r_c)$, and both endpoints are simple zeros. Take the smooth metric completion of $(r_b,r_c)\times\mathbb S^2$ with the data~\eqref{eq:radial-family} for these parameters, and set $D=\partial M$, $S=\varnothing$. These are the spatial data of the lukewarm Reissner--Nordstr\"om--de Sitter family; see, for example,
\cite[Section~III.A]{winstanley2008vacuum}.

At either endpoint, $1-m/r=ar$, so
\[
\dfrac{F'(r)}2=a(1-2ar).
\]
Consequently,
\[
\kappa_b=\dfrac{F'(r_b)}2=a\delta,\qquad \kappa_c=-\dfrac{F'(r_c)}2=a\delta.
\]
The quadruple is therefore critical despite having two Dirichlet boundary components. Here $|E|_g^2=m^2/r^4$ is nonconstant, and
\[
\lambda_1^D(-\Delta_g+|E|_g^2)=\Lambda=3a^2.
\]
\end{example}

\begin{example}[A charged purely Dirichlet example that is not critical]
\label{ex:unequal-horizons}
In \eqref{eq:radial-family}, take
\[
n=3,\qquad m=\dfrac9{16},\qquad Q=\dfrac12,\qquad \Lambda=\dfrac38.
\]
Then
\[
F(r)=1-\dfrac9{8r}+\dfrac1{4r^2}-\dfrac{r^2}{8}=-\dfrac{(r-1)(r-2)(r^2+3r-1)}{8r^2}>0\quad(1<r<2).
\]
Complete $(1,2)\times\mathbb S^2$ smoothly at the two horizons and put $D=\partial M$, $S=\varnothing$. This is a Robin electrostatic quadruple, but
\[
\kappa_1=\dfrac{F'(1)}2=\dfrac3{16},\qquad \kappa_2=-\dfrac{F'(2)}2=\dfrac9{64}.
\]
The two constants differ. Hence the quadruple is not critical under the variations in Corollary~\ref{cor:pure-dirichlet-electrostatic}. Thus purely Dirichlet electrostatic quadruples are not automatically critical when their boundary is disconnected.
\end{example}

\section{Appendix}\label{sec:appendix}

\subsection{Reissner--Nordstr\"om manifold}\label{sec:RN manifold}
In this section we make some computations for the Reissner--Nordstr\"om manifold in order to unify the notations in the papers~\cite{cederbaum2016uniqueness,jahns2019photon} and~\cite{cruz2024minmax}. 

We use the Reissner--Nordstr\"om data~\eqref{eq:rn-data} and the normalization $\mathfrak C_n$ fixed in the introduction. The associated spacetime metric is $\mathfrak g_{m,q}=-V_{m,q}^2dt^2+g_{m,q}$, and for $m\geqslant|q|$ the outer horizon is given by~\eqref{eq:rn-horizon}.

Under the isotropic coordinate change $r=s\varphi_{m,q}(s)$, with $\varphi_{m,q}$ defined in~\eqref{eq:rn-isotropic-factor}, the spatial metric becomes
\[
g_{m,q}=\varphi_{m,q}^2\delta,
\]
where $\delta$ is the Euclidean metric. Substituting this coordinate change into~\eqref{eq:rn-data} yields the expressions~\eqref{eq:rn-isotropic-fields} for the static potential and electric field stated in the introduction.

\begin{remark}\label{rmk:connect. E and Psi}
The electric-potential formulation of the Reissner--Nordstr\"om data in~\cite{cederbaum2016uniqueness,jahns2019photon} uses
\[
\Psi_{m,q}=\dfrac{q}{\mathfrak C_n r^{n-2}},\qquad
d\Psi_{m,q}=-\dfrac{(n-2)q}{\mathfrak C_n r^{n-1}}\,dr.
\]
Together with~\eqref{eq:rn-data}, these identities give $E_{m,q}=-\nabla^g\Psi_{m,q}/V_{m,q}$. In particular,
\[
|E_{m,q}|_{g_{m,q}}=\dfrac{(n-2)|q|}{\mathfrak C_n r^{n-1}},\qquad
\Psi_{m,q}^2=\psi^2|E_{m,q}|_{g_{m,q}}^2,\qquad \psi=\dfrac{r}{n-2}.
\]

Also, the equation for Reissner--Nordstr\"om electric potential $\Psi_{m,q}$ in isotropic coordinates is
\begin{align*}
    \Psi_{m,q}(s) = \dfrac{q}{\mathfrak{C}_n^2 s^{n-2}\left(1 + \dfrac{m+q}{2s^{n-2}}\right)\left(1 + \dfrac{m-q}{2s^{n-2}}\right)}.
\end{align*}
\end{remark}

\begin{remark}
Note that 
\begin{equation}\label{eq:ARNS_metric_funct_isotropic}
   \varphi_{m,q}=\left(\dfrac{(V_{m,q} +1)^2 -\mathfrak{C}_n^2\psi^2|E_{m,q}|_{g_{m,q}}^2}{4}\right)^{-\frac{1}{n-2}} = \left(\dfrac{(V_{m,q} +1)^2 -\mathfrak{C}_n^2\Psi_{m,q}^2}{4}\right)^{-\frac{1}{n-2}}.
\end{equation} 
\end{remark}

\subsection{Radial electrostatic data and photon spheres for arbitrary \texorpdfstring{$\Lambda$}{Lambda}}\label{ap:rps}
We consider the family~\eqref{eq:radial-family} in dimension $n\geqslant3$, with arbitrary cosmological constant $\Lambda$, on a radial interval $I\subset(0,\infty)$ where $F>0$. This includes the Reissner--Nordstr\"om data when $\Lambda=0$ and $Q=q$.

\emph{Arclength coordinates and the interior equations.}
Define $s$ by $ds=dr/\sqrt{F(r)}$, so $s$ increases with $r$. In this subsection, $s$ is arclength, rather than the isotropic coordinate of Appendix~\ref{sec:RN manifold}. A prime on $r(s)$ or $V(s)$ denotes differentiation with respect to $s$, whereas $F'(r)=dF/dr$ and $F''(r)=d^2F/dr^2$. Thus the chain rule gives
\begin{align}\label{eq:radial-arclength}
r'(s)&=\frac{dr}{ds}=\sqrt{F(r(s))}=V(s),\nonumber\\ 
r''(s)&=\frac{d}{ds}\sqrt{F(r(s))}=\frac{F'(r(s))}{2\sqrt{F(r(s))}}r'(s)=\frac{F'(r(s))}{2},\\
V'(s)&=\frac{F'(r(s))}{2},\qquad V''(s)=\frac{F''(r(s))}{2}r'(s)=\frac{F''(r(s))}{2}V(s).\nonumber
\end{align}
The coordinate vector transforms as $\partial_s=\sqrt F\,\partial_r$. Consequently, writing $a_n=\sqrt{(n-1)(n-2)/2}$ and $e(r)=a_nQ/r^{n-1}$, we have
\[
g=ds^2+r(s)^2g_{\mathbb S^{n-1}},\qquad E=e(r(s))\partial_s,\qquad |E|_g^2=e(r)^2.
\]

For completeness, direct differentiation of $F$ gives
\begin{align*}
F'(r)&=\frac{2(n-2)m}{r^{n-1}}-\frac{2(n-2)Q^2}{r^{2n-3}}-\frac{4\Lambda r}{n(n-1)},\\
F''(r)&=-\frac{2(n-2)(n-1)m}{r^n}+\frac{2(n-2)(2n-3)Q^2}{r^{2n-2}}-\frac{4\Lambda}{n(n-1)}.
\end{align*}
Hence
\begin{align}
\frac{F''}2+\frac{n-1}{2r}F'&=\frac{2(n-2)}{n-1}e(r)^2-\frac{2\Lambda}{n-1},\label{eq:radial-identity-1}\\
\frac{(n-2)(1-F)}{r^2}-\frac{F'}{2r}&=\frac{F'}{2r}+\frac{2}{n-1}\bigl(e(r)^2+\Lambda\bigr).
\label{eq:radial-identity-2}
\end{align}
To verify that these identities give the electrostatic equations, let $T$ be a $g$-unit vector tangent to a coordinate sphere. The warped-product formulas give
\begin{align*}
\Ric_g(\partial_s,\partial_s)&=-(n-1)\frac{r''}{r}=-\frac{n-1}{2r}F',\\
\Ric_g(T,T)&=\frac{(n-2)(1-(r')^2)-rr''}{r^2}=\frac{(n-2)(1-F)}{r^2}-\frac{F'}{2r},\\
\Hess_gV(\partial_s,\partial_s)&=V''=\frac{F''}{2}V,\\
\Hess_gV(T,T)&=\frac{r'}rV'=\frac{F'}{2r}V,\\
\Delta_gV&=V''+(n-1)\frac{r'}rV'=\left(\frac{F''}{2}+\frac{n-1}{2r}F'\right)V.
\end{align*}
All mixed radial--tangential components vanish. Equation~\eqref{eq:radial-identity-1} proves both the radial component of~\eqref{eq:electrostatic:a} and the trace equation~\eqref{eq:electrostatic:b}; equation~\eqref{eq:radial-identity-2} proves the tangential component of~\eqref{eq:electrostatic:a}. Finally,
\[
\mydiv_gE=\frac1{r^{n-1}}\frac{d}{ds}\bigl(r^{n-1}e(r(s))\bigr)=0,\qquad VE^\flat=a_n\frac{Q}{r^{n-1}}\,dr,
\]
and the latter one-form is closed. Thus all interior equations~\eqref{eq:electrostatic:a}--\eqref{eq:electrostatic:c} hold for every $\Lambda$.

\emph{Simple zeros and surface gravity.}
Let $r_0>0$ be a simple zero of $F$ bounding $I$, and let $s_0$ be the limiting arclength. The integral $\int dr/\sqrt F$ is finite at $r_0$. The equation $r''=F'(r)/2$, with $r(s_0)=r_0$ and $r'(s_0)=0$, is a regular smooth initial-value problem and yields
\begin{align*}
r(s)&=r_0+\frac{F'(r_0)}4(s-s_0)^2+O((s-s_0)^4),\\
V(s)=r'(s)&=\frac{F'(r_0)}2(s-s_0)+O((s-s_0)^3)
\end{align*}
on the appropriate side of $s_0$. The even extension of $r$ and the odd extension of $V$ are smooth, and so is $E=a_nQr^{1-n}\partial_s$. Therefore the completion adds a smooth nondegenerate Dirichlet boundary with
\begin{equation}\label{eq:radial-surface-gravity}
\kappa=|\nabla V|_g\big|_{r=r_0}=|V'(s_0)|=\frac{|F'(r_0)|}{2}.
\end{equation}

\emph{Robin spheres and photon-sphere radii.} For a coordinate sphere $\Sigma_r$, choose unit normal
\[
\nu=\varepsilon\sqrt{F(r)}\,\partial_r=\varepsilon\partial_s,\qquad \varepsilon\in\{-1,1\}.
\]
Since $\gamma_r=r^2g_{\mathbb S^{n-1}}$, differentiation in the normal direction gives
\begin{equation}\label{eq:mean curvature of ps in RN}
B_{\Sigma_r}=\frac12\mathcal L_\nu\gamma_r
=\varepsilon\frac{\sqrt F}{r}\gamma_r,\qquad
H_{\Sigma_r}=\varepsilon\frac{(n-1)\sqrt F}{r}.
\end{equation}
Similarly,
\begin{equation}\label{eq:derivative in the norm.dir. to V}
\partial_\nu V=\varepsilon\frac{F'(r)}2
=\varepsilon\left((n-2)\frac{m}{r^{n-1}}
-(n-2)\frac{Q^2}{r^{2n-3}}-\frac{2\Lambda r}{n(n-1)}\right).
\end{equation}
It follows directly, for either orientation, that the full tensorial Robin condition is
\begin{equation}\label{eq:radial-robin-condition}
(\partial_\nu V)\gamma_r-VB_{\Sigma_r}
=\varepsilon\left(\frac{F'}2-\frac Fr\right)\gamma_r=0
\quad\Longleftrightarrow\quad rF'(r)-2F(r)=0.
\end{equation}
As $V$ and the electric potential are constant on these spheres, this is also the photon-sphere condition in the associated static spacetime. Substitution of $F$ yields
\begin{align*}
rF'-2F
&=\frac{2(n-2)m}{r^{n-2}}-\frac{2(n-2)Q^2}{r^{2(n-2)}}
-\frac{4\Lambda r^2}{n(n-1)}\\
&\quad-2+\frac{4m}{r^{n-2}}-\frac{2Q^2}{r^{2(n-2)}}
+\frac{4\Lambda r^2}{n(n-1)}\\
&=-2+\frac{2nm}{r^{n-2}}-\frac{2(n-1)Q^2}{r^{2(n-2)}}.
\end{align*}
The cosmological terms cancel. Thus, setting $x=r^{n-2}>0$, we obtain
\begin{equation}\label{eq:photon-polynomial-lambda}
x^2-nmx+(n-1)Q^2=0,
\end{equation}
with candidate radii
\begin{equation}\label{eq:photon-radii-lambda}
r_{\mathrm{ph}}^\pm
=\left(\frac{nm\pm\sqrt{n^2m^2-4(n-1)Q^2}}{2}\right)^{1/(n-2)}.
\end{equation}
Only real, strictly positive roots lying in the chosen interval $I$ are admissible. In particular, $F(r_{\mathrm{ph}}^\pm)>0$ is essential: a root on a horizon does not define a Robin face with positive potential. Although the algebraic candidates are independent of $\Lambda$, their admissibility depends on $\Lambda$ and on the static interval. At a candidate radius,~\eqref{eq:photon-polynomial-lambda} also gives the useful test
\begin{equation}\label{eq:photon-admissibility-lambda}
F(r_{\mathrm{ph}})
=\frac{n-2}{n}\left(1-\frac{Q^2}{r_{\mathrm{ph}}^{2(n-2)}}\right)
-\frac{2\Lambda r_{\mathrm{ph}}^2}{n(n-1)}>0.
\end{equation}

For $\Lambda=0$ and $Q=q$, the outer candidate $r_{\mathrm{ph}}^+$ recovers the radius in~\eqref{eq:photon_sphere:a}, while the inner candidate is
\begin{equation}\label{eq:photon_sphere:b}
r_{\mathrm{ph}}^-=\left(\frac{nm-\sqrt{n^2m^2-4(n-1)q^2}}{2}\right)^{1/(n-2)}.
\end{equation}
In this asymptotically flat specialization, the sub-extremal and nonzero extremal exteriors have one photon sphere, at $r_{\mathrm{ph}}^+$. If $|q|>m>2\sqrt{n-1}|q|/n$, the super-extremal region has two, at $r_{\mathrm{ph}}^\pm$; at $m=2\sqrt{n-1}|q|/n>0$ they coincide. Below this positive threshold there are none; see~\cite[Corollary~3.2]{cederbaum2023equi}. For $\Lambda\neq0$, the appropriate count instead follows by applying~\eqref{eq:photon-admissibility-lambda} on the chosen static interval.

%%%%%%%%%%%%%%%%%%%%%%%%%%%%%%%%%%%%%%%%%%%%%%%%%%%%%%%%%%%%%%%%%%%%%%%%%
%%%  Bibliography
%%%%%%%%%%%%%%%%%%%%%%%%%%%%%%%%%%%%%%%%%%%%%%%%%%%%%%%%%%%%%%%%%%%%%%%%%
\bibliography{mybib}
\bibliographystyle{alpha}

\end{document}